\documentclass{cedram-aif}
\pdfoutput=1
\usepackage[pagebackref,bookmarks,bookmarksdepth=4]{hyperref}
\usepackage{amscd}

\usepackage[T1]{fontenc}

\hypersetup{
 pdfinfo={
   Title={Solvable Groups, Free Divisors and Nonisolated Matrix
Singularities I: Towers of Free Divisors},
   Author={James Damon and Brian Pike},
 }
}

\newcommand{\C}{{\mathbb C}}

\newcommand{\flushpar}{\par \noindent}

\newcommand{\dlog}{{\operatorname{Derlog}}}

\newcommand{\pd}[2]{\dfrac{\partial#1}{\partial#2}}

\def \bb {\mathbf {b}}

\def \bg {\mathbf {g}}
\def \bG {\mathbf {G}}

\def \bk {\mathbf {k}}

\def \bV {\mathbf {V}}

\def \b1 {\mathbf {1}}

\def \iti {\text{\it i}}
\def \itj {\text{\it j}}
\def \itm {\text{\it m}}

\def \cA {\mathcal{A}}

\def \cD {\mathcal{D}}
\def \cE {\mathcal{E}}

\def \cL {\mathcal{L}}

\def \cO {\mathcal{O}}

\def \cU {\mathcal{U}}
\def \cV {\mathcal{V}}

\def \gevar {\varepsilon}

\def \gl {\lambda}
\def \gs {\sigma}

\def \dim {{\rm dim}\,}

\def\C{\mathbb C}

\def\Sym{\mathrm{Sym}}
\def\Sk{\mathrm{Sk}}
\def\GL{\mathrm{GL}}
\def\SL{\mathrm{SL}}
\def\Diff{\mathrm{Diff}}

\def\Pf{\mathrm{Pf}}
\def\gl{\mathfrak{gl}}

\def\g{\mathfrak{g}}

\def\k{\mathfrak{k}}

\def\GL{\mathrm{GL}}

\equalenv{Thm}{theo}
\equalenv{Corollary}{coro}
\equalenv{Proposition}{prop}
\equalenv{Lemma}{lemm}
\equalenv{Conjecture}{conj}
\equalenv{Definition}{defi}
\equalenv{Example}{exem}  
\equalenv{Remark}{rema}

\numberwithin{equation}{section}

\title[Solvable Groups and Free Divisors]
{Solvable Groups, Free Divisors and Nonisolated Matrix Singularities I:
Towers of Free Divisors}

\alttitle{Groupes Solvables, Diviseurs Libres et Singularit\'es
des Matrices Nonisol\'ees I: Tours des Diviseurs Libres}

\author{\firstname{James} \lastname{Damon}}
\address{Department of Mathematics \\
University of North Carolina \\
Chapel Hill, NC 27599-3250 
(USA)}
\email{jndamon@math.unc.edu}
\thanks{JD was partially supported by the National Science Foundation grants
DMS-0706941 and DMS-1105470.}

\author{\firstname{Brian} \lastname{Pike}}
\address{Department of Computer and Mathematical Sciences \\
University of Toronto Scarborough \\
1265 Military Trail  \\
Toronto, ON M1C 1A4 
(Canada)}
\email{bapike@gmail.com}
\thanks{This paper contains work from BP's
Ph.D.~dissertation at Univ.~of North Carolina}

\altkeywords{Espaces vectoriels pr\'ehomog\`enes, diviseurs libres, diviseurs libres 
lin\'eaires, vari\'et\'es d\'eterminantales, vari\'et\'es de Pfaff, groupes 
alg\'ebriques solvables, factorisations du type Cholesky, repr\'esentations 
par blocs, vari\'et\'e des orbites exceptionelles, alg\`ebres de Lie 
solvables de dimension infinie}


\begin{document}
\begin{abstract}
We introduce a method for obtaining new classes of free divisors from 
representations $V$ of connected linear algebraic groups $G$ where $\dim 
G = \dim V$, with $V$ having an open orbit.  We give sufficient conditions 
that the complement of this open orbit, the \lq\lq exceptional orbit 
variety\rq\rq, is a free divisor (or a slightly weaker free* divisor) for 
\lq\lq block representations\rq\rq of both solvable groups and extensions 
of reductive groups by them.  These are representations for which the 
matrix defined from a basis of associated \lq\lq representation vector 
fields\rq\rq on $V$ has block triangular form, with blocks satisfying 
certain nonsingularity conditions.  
\par
For towers of Lie groups and representations this yields a tower of free 
divisors, successively obtained by adjoining varieties of singular 
matrices.  This applies to solvable groups which give classical 
Cholesky-type factorization, and a modified form of it, on spaces of 
$m \times m$ symmetric, skew-symmetric  or general matrices.  For 
skew-symmetric matrices, it further extends to representations of 
nonlinear infinite dimensional solvable Lie algebras. 
\end{abstract}

\begin{altabstract}
Nous introduisons une m\'ethode pour obtenir de nouvelles classes de 
diviseurs libres \`a partir des repr\'esentations $V$ de groupes 
alg\'ebriques lin\'eaires connexes $G$ pour lesquelles $\dim G = \dim V$ 
et $V$ a une orbite ouverte.  Nous donnons des conditions suffisantes pour 
lesquelles le compl\'ementaire de cette orbite ouverte, la \lq\lq 
vari\'et\'e des orbites exceptionelles\rq\rq, est une diviseur libre (ou un 
diviseur libre* plus faible) pour des \lq\lq repr\'esentations par 
blocs\rq\rq\, \`a la fois des groupes solvables et des extensions des 
groupes r\'eductifs par ces groupes.  Ce sont des repr\'esentations pour 
lesquelles la matrice d\'efinie \`a partir d\rq une base des \lq\lq champs 
des vecteurs associ\'es \rq\rq de la repr\'esentation $V$, a une forme 
triangulaire bloc et les blocs satisfont certaines conditions de 
non-singularit\'e.  
\par
Pour les tours des groupes de Lie et leurs repr\'esentations ce r\'esultat  
donne une tour de diviseurs libres obtenue en avoisinant successivement 
des vari\'et\'es de matrices singuli\`eres.  Il s\rq applique aux groupes 
solvables qui donnent la factorisation classique du type Cholesky, et une 
forme modifi\'ee de celle ci, sur les espaces des matrices $m \times m$
 sym\'etriques, antisym\'etriques, ou g\'en\'erales.  Pour les matrices 
antisym\'etriques, il s\rq \'etend aussi aux repr\'esentations des 
alg\`ebres de Lie solvables et non-lin\'eaires de dimension infinie. 
\end{altabstract}

\maketitle

\section*{Introduction}
\label{S:sec0}
\par
In this paper and part II \cite{DP}, we introduce a method for computing
the \lq\lq vanishing topology\rq\rq of nonisolated matrix singularities.  A
matrix singularity arises from a holomorphic germ $f_0 : \C^n, 0 \to M, 0$,
where $M$ denotes a space of matrices. If $\cV \subset M$ denotes the
variety of singular matrices, then we require that $f_0$ be transverse to
$\cV$ off $0$ in $\C^n$.  Then, $V_0 = f_0^{-1}(\cV)$ is the corresponding
matrix singularity.  Matrix singularities have appeared prominently in the
Hilbert--Burch theorem \cite{Hi}, \cite{Bh} for the representation of
Cohen--Macaulay singularities of codimension $2$ and for their
deformations by Schaps \cite{Sh}, by Buchsbaum-Eisenbud \cite{BE} for
Gorenstein singularities of codimension $3$, and in the defining support
for Cohen-Macaulay modules, see e.g. Macaulay\cite{Mc} and 
Eagon-Northcott \cite{EN}.  Considerable recent work has concerned the 
classification of various types of matrix singularities, including Bruce 
\cite{Br}, Haslinger \cite{Ha}, Bruce--Tari \cite{BrT}, and 
Goryunov--Zakalyukin \cite{GZ} and for Cohen--Macaulay singularities by 
Fr\"{u}hbis-Kr\"{u}ger--Neumer \cite{Fr} and \cite{FN}.\par
The goal of this first part of the paper is to use representation theory for 
connected solvable linear algebraic groups to place the variety of singular 
matrices in a geometric configuration of divisors whose union is a free 
divisor.  In part two, we then show how to use the resulting geometric 
configuration and an extension of the method of L\^e-Greuel \cite{LGr} to 
inductively compute the \lq\lq singular Milnor number\rq\rq of the matrix 
singularities  in terms of a sum of lengths of determinantal modules 
associated to certain free divisors (see \cite{DM} and \cite{D1}).  This 
will lead, for example, in part II to new formulas for the Milnor numbers 
of Cohen-Macaulay surface singularities.  Furthermore, the free divisors 
we construct in this way are distinguished topologically by both their 
complements and Milnor fibers being $K(\pi, 1)$\rq s \cite{DP2}.\par
In this first part of the paper, we identify a special class of 
representations of linear algebraic groups (especially solvable groups) 
which yield free divisors.  Free divisors arising from representations are 
termed \lq\lq linear free divisors\rq\rq by Mond, who with Buchweitz 
first considered  those that arise from representations of reductive 
groups using quivers of finite type \cite{BM}.  While reductive groups and 
their representations (which are completely reducible) are classified, this 
is not the case for either solvable linear algebraic groups nor their 
representations (which are not completely reducible).  We shall see that 
this apparent weakness is, in fact, an advantage.
\par
We consider an equidimensional (complex) representation of a connected 
linear algebraic group $\rho : G \to \GL(V)$, so that 
$\dim G = \dim V$, 
and for which the representation has an open orbit $\cU$.  Then, the
complement $\cE = V \backslash \cU$, the \lq\lq exceptional orbit
variety\rq\rq, is a hypersurface formed from the positive codimension
orbits.  We introduce the condition that the representation is a 
\lq\lq block representation\rq\rq, which is a refinement of the 
decomposition arising from the Lie-Kolchin theorem for solvable linear 
algebraic groups.  This is a representation for which the matrix 
representing a basis of associated vector fields on $V$ defined by the 
representation, 
using a basis for $V$, can be expressed  as a block triangular matrix, with 
the blocks satisfying certain nonsingularity conditions.  We use the Lie 
algebra structure of $G$ to identify the blocks and obtain a defining 
equation for $\cE$.  \par
In Theorem \ref{BlkRepFrDiv} we give a criterion that such a block
representation yields a linear free divisor and for a slightly weaker
version, we still obtain a free* divisor structure (where the exceptional
orbit variety is defined with nonreduced structure, see \cite{D3}).  We 
shall see more
generally that the result naturally extends to \lq\lq towers of
groups acting on a tower of representations\rq\rq to yield a tower of free
divisors in Theorem \ref{Towerthm}.  This allows us to inductively place
determinantal varieties of singular matrices within a free divisor by
adjoining a free divisor arising from a lower dimensional representation.
\par
We apply these results to representations of solvable linear algebraic 
groups associated to Cholesky-type  factorizations for the different types 
of complex matrices.  We show in  Theorem \ref{CholTow} that the 
conditions for the existence of 
Cholesky-type factorizations for the different types of complex matrices 
define the exceptional orbit varieties which are either
free divisors or free* divisors.  For those cases with only free* divisors,
we next introduce a modified form of Cholesky factorization which 
modifies the solvable groups to obtain free divisors still containing the
varieties of singular matrices.  This method extends to factorizations for 
$(n-1) \times n$ matrices (Theorem \ref{ModCholRep}).  \par
A new phenomena arises in \S \ref{S:NonlinLie} for skew-symmetric 
matrices.  We introduce a modification of a block representation which 
applies to infinite dimensional nonlinear solvable Lie algebras.  Such 
algebras are examples of \lq\lq holomorphic solvable Lie algebras\rq\rq 
not generated by finite dimensional solvable Lie algebras.  We again prove 
in Theorem \ref{Thm.nonlinLiealg} that the exceptional orbit varieties for 
these block representations are free divisors.  \par 
Moreover, in \S \ref{S:OpBLkReps} we give three operations on block 
representations which again yield block representations: quotient, 
restriction, and extension.  In \S \ref{S:RestrBlocReps} the restriction and 
extension operations are applied to block representations obtained from 
(modified) Cholesky-type factorizations to obtain auxiliary block 
representations which will play an essential role in part II in computing 
the vanishing topology of the matrix singularities.  \par
The representations we have considered so far for matrix singularities 
are induced from the simplest representations of $\GL_m(\C)$.  These 
results will as well apply to representations of solvable linear algebraic 
groups obtained by restrictions of representations of reductive groups to 
solvable subgroups and extensions by solvable groups.  These results are 
presently under investigation.  \par
We wish to thank the referee for the careful reading of the paper and the 
number of useful suggestions which he made.  

\section{Preliminaries on Free Divisors Arising from Representations of 
Algebraic Groups}
\label{S:sec1}
\par
Our basic approach uses 
hypersurface germs $\cV, 0 \subset \C^p, 0$ that are free divisors in the 
sense of Saito \cite{Sa}, and his corresponding criteria.
\subsection*{Free Divisors and Saito's Criteria} \hfill 
\par  
Quite generally if $I(\cV)$ is the defining ideal for a hypersurface germ 
$\cV, 0 \subset \C^p, 0$, we let
$$  \dlog(\cV) \,\, = \,\, \{ \zeta \in \theta_p :  \makebox{ such that }
\zeta (I(\cV)) \subseteq I(\cV) \}  $$
where $\theta_p$ denotes the module of germs of holomorphic vector 
fields on $\C^p, 0$.  Saito \cite{Sa} defines $\cV$ to be a {\it free divisor} 
if $\dlog(\cV)$ is a free $\cO_{\C^p, 0}$-module (necessarily of rank $p$).  
\par
Saito also gave two fundamental criteria for establishing that a 
hypersurface germ $\cV, 0 \subset \C^p, 0$ is a free divisor.  Suppose 
$\zeta_{i} \in \theta_p$ for $i = 1, \dots , p$.  Then, for coordinates 
$(y_1, \dots , y_p)$ for $\C^p, 0$, we may write a basis

\begin{equation}
\label{Eqn1.7}
     \zeta_{i} \,\, =  \,\, \sum_{j = 1}^{p} a_{j, i} \pd{ }{y_j}   \qquad  i = 1,
\dots ,p
\end{equation}
with $a_{j, i} \in \cO_{\C^p, 0}$.
We refer to the $p \times p$ matrix $A = (a_{j, i})$ as a {\it coefficient
matrix} or \emph{Saito matrix} for the vector fields $\{
\zeta_{i}\}$, and we call the determinant $\det(A)$ the {\it coefficient 
determinant}.  
\par
A sufficient condition that $\cV, 0$ is a free divisor is given by 
Saito\rq s criterion \cite{Sa} which has two forms.
\begin{Thm}[Saito's criterion] \hfill
\label{Saitocrit}
\begin{enumerate}
\item
The hypersurface germ $\cV, 0 \subset \C^p, 0$ is a free divisor if there
are $p$ elements $\zeta_1, \dots , \zeta_p \in \dlog(\cV)$ and a basis
$\{ w_j\}$ for $\C^p$ so that the coefficient matrix $A = (a_{i\, j})$ has
determinant which is a reduced defining equation for $\cV, 0$.
Then, $\zeta_1, \dots , \zeta_p$ is a free module basis for $\dlog(\cV)$.
\par
Alternatively,
\item Suppose the set of vector fields $\zeta_1,\dots ,
\zeta_p$ is closed under Lie bracket, so that for all $i$ and $j$
$$  [ \zeta_i, \zeta_j ]  \,\, = \,\, \sum_{k = 1}^{p} h_k^{(i, j)} \zeta_k  $$
for $h_k^{(i, j)} \in \cO_{\C^p, 0}$.  If the coefficient determinant is a 
reduced defining equation for a hypersurface germ $\cV, 0$, then $\cV, 0$ 
is a free divisor and $\zeta_1,\dots , \zeta_p$ form a free module basis 
of $\dlog(\cV)$.
\end{enumerate}
\end{Thm}
\par
We make several remarks regarding the definition and criteria.  First, in 
the case of a free divisor $\cV, 0$, there are two choices of bases 
involved in the definition, the basis $\pd{ }{y_i}$ and the set of generators 
$\zeta_1, \dots , \zeta_p$.  Hence the coefficient matrix is highly 
nonunique.  However, the coefficient determinant is well-defined up to 
multiplication by a unit as it is a generator for the $0$-th Fitting ideal of 
the quotient module $\theta_p/\dlog(\cV)$.  
Second, $\dlog(\cV)$ is more than a just finitely generated module over 
$\cO_{\C^p, 0}$; it is also a Lie algebra. However, with the exception of 
the $\{\zeta_i\}$ being required to be closed under Lie bracket in the 
second criteria, the Lie algebra structure of $\dlog(\cV)$ does not enter 
into consideration.  \par
In Saito\rq s second criterion, if we let $\cL$ denote the 
$\cO_{\C^p, 0}$--module generated by $\{ \zeta_i, i = 1, \dots , p\}$, then 
$\cL$ is also a Lie algebra.  More generally we shall refer to any finitely 
generated $\cO_{\C^p, 0}$--module $\cL$ which is also a Lie algebra as a 
{\it (local) holomorphic Lie algebra}.  We will consider holomorphic Lie 
algebras defined for certain distinguished classes of representations of 
linear algebraic groups and use the Lie algebra structure to show that the 
coefficient matrix has an especially simple form.  \par  
\subsection*{Prehomogeneous Vector Spaces and Linear Free Divisors} 
\hfill  
\par
Suppose that $\rho : G \to \GL(V)$ is a rational representation of a 
connected complex linear algebraic group.  If there is an open orbit $\cU$ 
then such a space with group action is called a prehomogeneous vector 
space and 
has been studied by Sato and Kimura \cite{So}, \cite{SK}, \cite{K} but 
from the point 
of view of harmonic analysis.  They have effectively determined the 
possible prehomogeneous vector spaces arising from irreducible
representations of reductive groups.  \par 
If $\g$ denotes the Lie algebra of $G$, then for each $X \in \g$, there is a 
vector field on $V$ defined by 
\begin{equation}
\label{Eqn1.6}
  \xi_{X}(v) \,\, = \,\,  \pd{ }{t}(\exp(t\cdot X)\cdot v)_{| t = 0} \qquad 
\makebox{ for } v \in V \, .
\end{equation}
In the case $\dim G = \dim V = n$, Mond observed that if $\{ X_i\}_{i = 
1}^{n}$ is a basis of the Lie algebra $\g$ and the coefficient matrix of 
these vector fields with respect to coordinates for $V$ has reduced 
determinant, then Saito\rq s criterion can be applied to conclude $\cE = V 
\backslash \cU$ is a free divisor with $\dlog(\cE)$ generated by the $\{ 
\xi_{X_i}, i = 1, \dots , n\}$.  This idea was applied by Buchweitz--Mond to 
reductive groups arising from quiver representations of finite type 
\cite{BM}, more general quiver representations in \cite{GMNS} and
\cite{GMS}, and
irreducible representations of reductive groups in \cite{GMS}. 
In the 
case that $\cE$ is a free divisor, we follow Mond and call it a {\it linear 
free divisor}.  \par
We shall call a representation with $\dim G = \dim V$ an {\it 
equidimensional representation}.  Also, the variety $\cE = V \backslash 
\cU$ has been called the singular set or discriminant.  We shall be 
considering in part II mappings into $V$, which also have singular sets and 
discriminants.  To avoid confusion, we shall refer to $\cE$, which is the 
union of the orbits of positive codimension, as the {\it exceptional orbit 
variety}.  
\begin{Remark}
In the case of an equidimensional representation with open orbit, if there
is a basis $\{ X_i\}$ for $\g$ such that the determinant of the coefficient
matrix defines $\cE$ but with nonreduced structure, then we refer to
$\cE$ as being a {\em linear free* divisor}.  A free* divisor structure can 
still be used for determining the topology of nonlinear sections as is done 
in \cite{DM}, except correction terms occur due to the presence of 
``virtual singularities'' (see \cite{D3}).  However, by \cite{DP2}, the 
free* divisors that occur in this paper will have complements and Milnor 
fibers with the same topological properties as free divisors.
\end{Remark}
\par
In contrast with the preceding results, we shall be concerned with 
non-reductive groups, and especially connected solvable linear algebraic 
groups.  The representations of such groups $G$ cannot be classified as in 
the reductive case.  Instead, we will make explicit use of the Lie algebra 
structure of the Lie algebra $\g$ and special properties of its 
representation on $V$.  We do so by identifying it with its image in 
$\theta(V)$, which denotes the $\cO_{V, 0}$-module of germs of 
holomorphic vector fields on $V, 0$, which is also a Lie algebra.  We will 
view it as the Lie algebra of the group $\Diff(V, 0)$ of germs of 
diffeomorphisms of $V, 0$, even though it is not an infinite dimensional 
Lie group in the usual sense.  \par
 Let $\xi \in \itm \cdot \theta(V)$, with 
$\itm$ denoting the maximal ideal of $\cO_{V, 0}$.  Integrating $\xi$ 
gives a local one-parameter group of diffeomorphism germs $\varphi_t : 
V, 0 \to V, 0$ defined for $| t | < \gevar$ for some $\gevar > 0$,  which 
satisfy $\pd{\varphi_t}{t} = \xi\circ \varphi_t$ and
$\varphi_0 = id$.  Because we are only interested in germs for $t$ near 
$0$, we only need to consider the \lq\lq exponential map\rq\rq defined in 
terms of local one--parameter subgroups:
\begin{equation*}
\exp_{\xi} : (-\epsilon,\epsilon) \to \Diff (V, 0)   \qquad \makebox{ where 
} \quad \exp_\xi(t) = \varphi_{t} \, .
\end{equation*}
\par  
Second, we have the natural inclusion $\iti : \GL(V) \hookrightarrow 
\Diff(V, 0)$, where a linear transformation $\varphi$ is viewed as a germ 
of a diffeomorphism of $V, 0$.  There is a corresponding map
\begin{align}
    \tilde \iti :  \gl(V) &\longrightarrow  \itm \cdot \theta(V)      \\
A &\mapsto \xi_A   \notag
\end{align}
where the $\xi_A (v) = A(v)$ are 
\emph{linear vector fields}, whose coefficients are linear functions.  
Then, 
$\tilde \iti$ is a bijection between $\gl(V)$ and the subspace of linear 
vector fields.  A straightforward calculation shows that $\tilde \iti$ is a 
Lie algebra homomorphism provided we use the negative of the usual Lie 
bracket for $\itm \cdot \theta(V)$. \par
Given a representation $\rho : G \to \GL(V)$ of a (complex) connected  
linear algebraic group $G$ with associated Lie algebra homomorphism 
$\tilde{\rho}$, there is the following commutative exponential diagram. 
\par
\vspace{1ex}
\flushpar
{\it Exponential Diagram for a Representation} \par
\begin{equation}
\label{CD.assoc.vfs}
\begin{CD}
{\g} @>{\tilde{\rho}}>>  { \gl(V) } @>{\tilde{\iti}}>>
{\itm \cdot \theta(V)}\\
@V{\exp}VV     @V{\exp}VV      @V{\exp}VV   \\
{G}  @>{\rho}>>  {\GL(V)}  @>{\iti}>>  {\Diff(V, 0)}
\end{CD}
\end{equation}
\flushpar
where the exponential map for Lie groups is also viewed as a map to local 
one-parameter groups $X \mapsto \exp(t\cdot X)$.
\par
If $\rho$ has finite kernel, then $\tilde \rho $ is injective.  Even though it
is not standard, we shall refer to such a representation as a {\it faithful
representation}, as we could always divide by the finite group and obtain
an induced representation which is faithful and does not alter the
corresponding Lie algebra homomorphisms.  Hence, $\tilde \iti \circ \tilde
\rho$ is an isomorphism from $\g$ onto its image, which we shall denote
by $\g_V$.  \par
Hence, $\g_V \subset \itm \cdot \theta(V)$ has exactly the same Lie
algebra theoretic properties as $\g$.  For $X \in \g$, we slightly abuse 
notation by more simply denoting $\xi_{\tilde \rho (X)}$ by $\xi_{X} \in 
\g_V$, which we refer to as the associated {\it representation vector 
fields}.  
The $\cO_{V, 0}$ --module generated by $\g_V$ is a holomorphic Lie 
algebra which has as a set of generators $\{ \xi_{X_i}\}$, as $X_i$ varies 
over a basis of $\g$.  Saito\rq s criterion applies to the $\{ \xi_{X_i}\}$; 
however, we shall use the correspondence with the Lie algebra properties 
of $\g$ to deduce the properties of the coefficient matrix.  \flushpar
{\bf Notational Convention: } We will denote vectors in the Lie algebra 
$\g$ by $X_i$ and vectors in the space $V$ by either $u_i$, $v_i$, or 
$w_i$.  \par

\flushpar
\subsection*{Naturality of the Representation Vector Fields}  \hfill
\par
The naturality of the exponential diagram leads immediately to the 
naturality of the construction of representation
vector fields.  Let $\rho : G \to \GL(V)$ and $\rho^{\prime} : H \to \GL(W)$ 
be
representations of linear algebraic groups.  Suppose there is a Lie group
homomorphism $\varphi : G \to H$  and a linear transformation
$\varphi^{\prime} : V \to W$ such that when we view $W$ as a $G$
representation via $\varphi$, then $\varphi^{\prime}$ is a homomorphism
of $G$-representations.  We denote this by saying that $\Phi= (\varphi,
\varphi^{\prime}) : (G, V) \to (H, W)$ is {\it homomorphism of groups and
representations}.
\begin{Proposition}
\label{Prop1.9}
The construction of representation vector fields is natural in the sense
that if $\Phi= (\varphi, \varphi^{\prime}) : (G, V) \to (H, W)$ is a
homomorphism of groups and representations, then for any $X \in \g$,
the representation vector fields $\xi_X$ for $G$ on $V$ and $\xi_{\tilde
\varphi(X)}$ for $H$ on $W$ are $\varphi^{\prime}$--related.
\end{Proposition}
\begin{proof}
By (\ref{Eqn1.6}), for $v \in V$
\begin{align}
\label{Eqn1.8}
  d \varphi^{\prime}_v(\xi_{X}(v)) \,\, &= \,\,  \pd{
}{t}(\varphi^{\prime}(\exp(t\cdot X)\cdot v))_{| t = 0}
\,\, = \,\,  \pd{ }{t}(\varphi(\exp(t\cdot X)) \cdot \varphi^{\prime}(v))_{| t 
= 0}
\notag  \\
&= \,\,  \pd{ }{t}(\exp(t\cdot \tilde\varphi(X)) \cdot 
\varphi^{\prime}(v))_{| 
t = 0}
\,\, = \,\,  \xi_{\tilde \varphi(X)}(\varphi^{\prime}(v)) .
\end{align}
Hence, $\xi_{X}$ and $\xi_{\tilde \varphi(X)}$ are
$\varphi^{\prime}$--related as asserted.
\end{proof}
\par

\section{Block Representations of Linear Algebraic Groups}
\label{S:BlkReps}
We consider representations $V$ of connected linear algebraic groups $G$ 
which need not be reductive.  These may not be completely reducible; 
hence, there may be invariant subspaces $W \subset V$ without invariant 
complements.  It then follows that we may represent the elements of $G$ 
by block upper triangular matrices; however, importantly, it does not 
follow that the corresponding coefficient matrix for a basis of 
representation vector fields need be block triangular nor that the diagonal 
blocks need be square.  \par
There is a condition which we identify, which will lead to this stronger 
property and be the  basis for much that follows.  To explain it, we first 
examine the form of the representation vector fields for $G$.  We 
choose a basis for $V$ formed from a basis $\{w_i\}$ for the invariant 
subspace $W$ and 
a complementary basis $\{ u_j\}$ to $W$.  
\begin{Lemma}
\label{Lem4.1}
In the preceding situation, 
\begin{itemize}
\item[i)] any representation vector field  $\xi_X \in \g_V$ has
the form
\begin{equation}
\label{Eqn4.1}
  \xi_X \,\,  =  \,\, \sum_{\ell} b_{\ell} u_{\ell} \, + \,  
\sum_{j} a_{j} w_j
\end{equation}
where $a_{j} \in \cO_{V, 0}$ and $b_{\ell} \in \pi^*\cO_{V/W, 0}$
for $\pi : V \to V/W$ the natural projection;
\item[ii)]
if $G$ is connected, the representation of $G$ on $V/W$ is the trivial 
representation if and only if for each $\xi_{X} \in \g_V$, the coefficients 
$b_{\ell} = 0$ in \eqref{Eqn4.1}.  
\end{itemize}
\end{Lemma}
\begin{proof}
First, we know $( id, \pi) : (G, V) \to (G, V/W)$ is a homomorphism of
groups and representations.  By Proposition \ref{Prop1.9},
the representation vector fields $\xi_X$ on $V$ and
$\xi^{\prime}_X$ on $V/W$ for $X \in \g$  are $\pi$--related.  Hence, for 
i), the representation vector field $\xi^{\prime}_X$ on $V/W$ has the form 
of the first sum on the RHS of (\ref{Eqn4.1}).  The coefficients for the 
$w_j$ will be function germs in $\cO_{V, 0}$.  \par
For ii), if $G$ acts trivially on $V/W$ then for $X \in \g$, $\xi_X$ on $V$ 
is $\pi$-related to $\xi^{\prime}_X$ on $V/W$, whose one parameter 
subgroup is the identity.  Hence, $\xi^{\prime}_X$ on $V/W$ is $0$, so the 
$b_{\ell} = 0$.  Conversely, if each $\xi_X$ on $V$ has the form 
\eqref{Eqn4.1} with the coefficients $b_{\ell} = 0$, then $\xi_X$ is 
$\pi$-related to $\xi^{\prime}_X = 0$.  Thus, the one parameter group 
generated by $X$ on $V/W$ is the identity.  As this is true for all $X \in 
\g$, it follows that the exponential map has image in the identity 
subgroup.  Thus, a neighborhood of the identity of $G$ acts trivially on 
$V/W$, hence so does $G$ by the connectedness of $G$.
\end{proof}
\par
Next we introduce a definition.
\begin{Definition}
Let $G$ be a connected linear algebraic group which acts on $V$ and which 
has a $G$--invariant subspace $W \subset V$ with $\dim W = \dim G$ 
such that $G$ acts trivially on $V/W$.  We say that $G$ has a 
{\em relatively open orbit in $W$} if there is an orbit of $G$ in $V$ whose 
generic projection onto $W$ is Zariski open.
\end{Definition}
\par
This condition can be characterized in terms of the representation vector
fields of $G$.  We choose a basis $\{ \xi_{X_i} : i = 1, \dots , k\}$ for
$\g_V$, with $k=\dim(W)=\dim(G)$.
Then, as $G$ acts trivially on $V/W$, by Lemma \ref{Lem4.1} it
follows that we can write 
\begin{equation}
\label{Eqn4.3}
 \xi_{X_i} \,\,  =  \,\,  \sum_{j} a_{j i} w_j
\end{equation}
where $a_{j i} \in \cO_{V, 0}$.  We refer to the matrix $(a_{j i})$ as a {\it
relative coefficient matrix} for $G$ and $W$.  We also refer to $\det(a_{j
i})$ as the {\it relative coefficient determinant for $G$ and $W$}.\par
We note that the composition of the projection $V\to W$ with the
orbit map $G\to G\cdot v$
is a rational map.
Since $G$ acts trivially on $V/W$, $G\cdot v\subseteq v+W$.
Then, $G$ has a relatively open orbit in $W$ 
if and only if $G$ has an open orbit in $v+W$ for some $v\in V$.  Since the 
orbit through $v$ has tangent space spanned by the set $\{ \xi_{X_i}(v) : i 
= 1, \dots , k\}$, the image is Zariski open if and only if $\det(a_{j i})$ 
is nonzero at $v$.
We conclude
\begin{Lemma}
\label{Lem4.2}
The action of $G$ on $V$ has a relatively open orbit in $W$ if and only if 
the relative coefficient determinant is not zero. \qed
\end{Lemma}
We also note that the relative coefficient determinant is also 
well-defined up to multiplication by a unit, as by (\ref{Eqn4.3}) it is a 
generator for the $0$-th Fitting ideal for the quotient module $\cO_{V, 
0}\cdot \theta_{W, 0}/\cO_{V, 0}\cdot \g_V$. 
\par 
Now we are in a position to introduce a basic notion for us, that of a block
representation.
\begin{Definition}
\label{Def4.4} 
An equidimensional representation $V$ of a connected linear algebraic 
group $G$
will be called a {\em block representation} if:
\begin{itemize}
\item[i)]  there exists a sequence of $G$-invariant subspaces
$$   V = W_k \supset W_{k-1} \supset \cdots  \supset W_{1}  \supset
W_{0} = (0) .  $$
\item[ii)] for the induced representation $\rho_j : G \to \GL(V/W_j)$, we
let $K_j = \ker(\rho_j)$; then for all $j$, $\dim K_j = \dim W_j$ and the
action of $K_j/K_{j-1}$ on $V/W_{j-1}$ has a relatively open orbit in 
$W_{j}/W_{j-1}$.
\item[iii)]  the relative coefficient determinants $p_j$ for the
representations $K_j/K_{j-1}\to\GL(V/W_{j-1})$ and subspaces 
$W_j/W_{j-1}$ are all reduced and relatively prime in pairs in 
$\cO_{V, 0}$ (by Lemma \ref{Lem4.1}, $p_j \in \cO_{V/W_{j-1}, 0}$ 
and we obtain $p_j \in \cO_{V, 0}$ via pull-back by
the projection map from $V$).
\end{itemize}
We also refer to the decomposition of $V$ using the $\{ W_j\}$ and $G$ by
the $\{ K_j\}$ with the above properties as the {\em  decomposition for 
the block representation}.  Along with i) in the definition, we note there is 
a corresponding sequence of subgroups 
$$   G = K_k \supset K_{k-1} \supset \cdots  \supset K_{1}  \supset
K_{0} \,  .  $$
\par
Furthermore, if each $p_j$ is irreducible, then we will refer to it as a
{\em maximal block representation}.  \par
If in the preceding both i) and ii) hold, and the relative coefficient
determinants are nonzero but may be nonreduced or not relatively prime in
pairs, then we say that it is a {\em nonreduced block representation}.
\end{Definition}
\par

\subsection*{Block Triangular Form}\hfill\par
We deduce for a block representation $\rho : G \to \GL(V)$ (with
subspaces and kernels as in Definition \ref{Def4.4}) a special block 
triangular form for its coefficient matrix with respect to bases 
respecting the invariant subspaces $W_j$ and the  corresponding kernels 
$K_j$.  \par
Specifically, we first choose a  basis $\{w^{(j)}_i\}$ for
$V$ such that $\{ w^{(j)}_1, \dots , w^{(j)}_{m_j}\}$ is a complementary
basis to $W_{j-1}$ in $W_j$, for each $j$.  Second, letting $\bk_j$ denote
the Lie algebra for $K_j$,  we choose a basis $\{X^{(j)}_i\}$ for $\g$ such
that $\{ X^{(j)}_1, \dots , X^{(j)}_{m_j}\}$ is a complementary basis to
$\bk_{j-1}$ in $\bk_j$.  then we obtain (partially) ordered bases 
\begin{equation}
\label{Eqn4.5a}
 \{ w^{(k)}_{1}, \cdots, w^{(k)}_{m_k}, \cdots, \cdots, w^{(1)}_{1}, 
\cdots w^{(1)}_{m_1} \} 
\end{equation}
for $V$, and 
\begin{equation}
\label{Eqn4.5b}
 \{ X^{(k)}_{1}, \cdots, X^{(k)}_{m_k}, \cdots,  \cdots, X^{(1)}_{1}, \cdots,
X^{(1)}_{m_1}\} 
\end{equation}
for $\g$.  These bases have the property that the subsets $\{ w^{(j)}_{i}: 
1\leq j \leq \ell, 1\leq i\leq m_j \}$ form bases for the subspaces 
$W_{\ell}$, and the subsets $\{ X^{(j)}_{i}: 1\leq j \leq \ell, 1\leq i\leq 
m_j \}$, for the Lie algebras $\bk_{\ell}$ of kernels $K_{\ell}$.

\begin{Proposition}
\label{Prop4.5}
Let $\rho : G \to \GL(V)$ be a block representation with the ordered bases 
for $\g$ and $V$ given by (\ref{Eqn4.5a}) and (\ref{Eqn4.5b}).  Then, the 
coefficient matrix $A$ has a lower block triangular form as in 
(\ref{matr4.5}), where each $D_j$ is a $m_j\times m_j$ matrix. \par
Then, $p_j = \det(D_j)$ are the relative coefficient determinants.
\end{Proposition}
\begin{equation}
\label{matr4.5}
A=
\begin{pmatrix}
D_k &0 & 0 & 0 & 0 \\ *  & D_{k-1} & 0 & 0 & 0 \\  *  & * & \ddots & 0 & 0
\\
*  & * &* & \ddots & 0 \\  *  & *  & *  & * & D_1
\end{pmatrix},
\end{equation}

In (\ref{matr4.5}) if $p_1 = \det(D_1)$ is irreducible, then we will refer
to the variety $\cD$ defined by $p_1$ as the {\it generalized determinant
variety} for the decomposition. \par
As an immediate corollary we have

\begin{Corollary}
\label{Cor4.6}
For a block representation, the number of irreducible components in the
exceptional orbit variety is at least the number of diagonal blocks in the
corresponding block triangular form, with equality for a maximal block
representation.
\end{Corollary}

\begin{proof}[Proof of Proposition \ref{Prop4.5}]
Since $K_\ell$ acts trivially on $V/W_\ell$, by Lemma \ref{Lem4.1}, for 
$X\in \bk_\ell$ the associated representation vector field may be written 
as
\begin{equation}
\label{eqn:repvf2}
\xi_X =\sum_{j=1}^\ell \sum_{i=1}^{m_j} a_{i j} w_i^{(j)} 
\end{equation}
where as mentioned above, the basis for $W_\ell$ is given by $\{w_i^{(j)} : 
1\leq j \leq \ell, 1\leq i\leq m_j\}$.
Thus, for $\{ X_i^{(\ell)} : i = 1, \dots, m_j\}$  a complementary basis to 
$\bk_{\ell-1}$ in $\bk_{\ell}$, the columns corresponding to 
$\xi_{X_i^{(\ell)}}$ will be zero above the block $D_{\ell}$ as indicated.  
\par
Furthermore, the quotient maps $(\varphi,\varphi^{\prime}) : (K_\ell, 
V)\to (K_\ell/K_{\ell-1},V/W_{\ell-1})$ define a homomorphism of groups 
and representations.  Thus, again by Lemma \ref{Lem4.1}, the coefficients 
$a_{j\, i}$ of $w_j^{(\ell)}$, $j = 1, \dots , m_{\ell}$, for the 
$\xi_{X_i^{(\ell)}}$ are the same as those for the representation of 
$K_\ell/K_{\ell-1}$ on $V/W_{\ell-1}$.  Thus, we obtain $D_{\ell}$ as the 
relative coefficient matrix for $K_{\ell}/K_{\ell-1}$ and $V/W_{\ell-1}$.  
Thus $p_{\ell}=\det(D_{\ell})$.
\end{proof}
\par

\begin{Remark}
By Lemma \ref{Lem4.1} and Proposition \ref{Prop4.5}, in \eqref{matr4.5}
the entries of $D_\ell$ and $p_\ell=\det(D_\ell)$ are polynomials in 
$\cO_{V/W_{\ell-1}}$ that may be pulled back by the quotient maps to 
give elements of $\cO_V$.  Then, in coordinates $(x_i^{(j)})$ defined via 
the basis $\{ w_i^{(j)}\}$, we have $p_\ell\in R_\ell$,
which is the subring of $\cO_{V}$ 
generated by $\{x_i^{(j)} : \ell\leq j \leq k, 1\leq i\leq m_j\}$.  There is 
the resulting reverse sequence of subrings
$$ \cO_V=R_1\supset R_2\supset\cdots \supset R_k\, ; $$
and only $D_1$ and $p_1$ have entries in $\cO_V$.  Then, it follows that 
the determinant variety defined by $p_1$ can be 
``completed'' to the free divisor given by the exceptional orbit 
variety by adjoining the pull-back of the variety defined by $\prod_{i = 
2}^{k} p_i$.  We use this idea in \S \ref{S:TowReps} to use towers of block 
representations to inductively place such determinantal varieties between 
two free divisors.
\end{Remark}
\par
\begin{Remark}
\label{Rem4.6}
We also remark that there is a converse to Proposition \ref{Prop4.5} that 
given the invariant subspaces, associated kernels, and complementary 
bases, then the 
coefficient matrix is block lower triangular;
if also the diagonal blocks are square 
and the determinants $p_j$ are nonvanishing, then the dimension condition 
in Definition \ref{Def4.4} is satisfied and each induced
representation has a relatively open orbit.

\end{Remark}
\par
\subsection*{Exceptional Orbit Varieties as Free and Free* Divisors} 
\hfill
\par
We can now easily deduce from Proposition \ref{Prop4.5} the basic result
for obtaining linear free divisors from representations of linear algebraic 
groups.
\begin{Thm}
\label{BlkRepFrDiv}
Let $\rho : G \to  \GL(V)$ be a block representation of a connected linear
algebraic group $G$, with relative coefficient determinants $p_j, j = 1,
\dots, k$. Then, the exceptional orbit variety $\cE, 0 \subset
V, 0$ is a linear free divisor with reduced defining equation 
$\prod_{j =1}^{k} p_j = 0$  \par
If instead $\rho : G \to  \GL(V)$ is a nonreduced block representation, then
$\cE, 0 \subset V, 0$ is a linear free* divisor and $\prod_{j = 1}^{k} p_j =
0$ is a nonreduced defining equation for $\cE, 0$.
\end{Thm}
\begin{proof}
By Proposition \ref{Prop4.5}, we may choose bases for $\g$ and $V$ so
that the coefficient matrix has the form (\ref{matr4.5}).  Then, by the
block triangular form, the coefficient determinant equals $\prod_{j = 
1}^{k} p_j$,which by condition iii) for block representations is reduced.
As $\rho$ is algebraic, for  $v\in V$, the orbit map $G \to G\cdot v 
\subset V$ is rational so the orbit of $v$ is Zariski open if and only if
the orbit map at $v$ is a submersion, which is true if and only if the 
coefficient determinant is nonzero at $v$. As these Zariski open orbits are 
disjoint, there can be only one.  Hence, its complement is the exceptional 
orbit variety $\cE$ defined by the vanishing of the coefficient 
determinant.  \par
Since the representation vector fields belong to $\dlog (\cE)$, the first 
form of Saito\rq s Criterion (Theorem \ref{Saitocrit}) implies that $\cE$ 
is a free divisor.  \par
In the second case, if either the determinants of the relative coefficient
matrices $p_j$ are either nonreduced or not relatively prime in pairs then,
although $\prod_{j = 1}^{k} p_j = 0$ still defines $\cE$, it is nonreduced.
Hence, $\cE$ is then only a linear free* divisor.
\end{proof} 
\par
The usefulness of this result comes from several features: its general 
applicability to non-reductive linear algebraic groups, especially solvable 
groups; the behavior of block representations under basic operations 
considered in \S \ref{S:OpBLkReps}; the simultaneous and inductive 
applicability to a tower of groups and corresponding representations in \S 
\ref{S:TowReps}; and most importantly for applications, the abundance of 
such representations especially those appearing in complex versions of 
classical Cholesky-type factorization theorems \S \ref{S:CholFact}, 
their modifications \S \ref{S:ModCholreps}, \S \ref{S:NonlinLie}, and 
restrictions \S \ref{S:RestrBlocReps}.

\begin{Example}
There is a significant contrast to be made between the coefficient 
matrices for representations of reductive groups versus those for block 
representations in the non-completely reducible case.  For example, for an 
irreducible representation of a reductive group, such as is the case for 
quiver representations of finite type studied by Buchweitz-Mond  
\cite{BM}, it is not possible to represent the coefficient matrix in 
block lower triangular form, except as given by a single block.
Hence, the components of 
the exceptional orbit variety are not directly revealed by the structure of 
the coefficient matrix.
 
More generally, consider the action of $G = \prod_{i = 1}^{m} G_i$ on $V = 
\prod_{i = 1}^{m} V_i$ induced by the product representation, where each 
$G_i$ is reductive and $V_i$ irreducible, with $\dim(G_i)=\dim(V_i)$ and 
$G_i$ having an open orbit in $V_i$ for all $i$. This defines a nonreduced 
block 
representation with $W_j = \prod_{i = 1}^{j} V_i$  and $K_j = \prod_{i = 
1}^{j} G_i$, and the coefficient matrix is just block diagonal.
If each action of $G_i$ on $V_i$ defines a linear free divisor $\cE_i$, then 
$G$ acting on $V$ defines a linear free divisor which is a product union of 
the $\cE_i$ in the sense of \cite{D2}.  However, again the structure of the 
individual $\cE_i$ is not revealed by the block structure.  By contrast, as 
we shall see in subsequent sections, the block representations in the non-
completely reducible case, especially for representations of solvable 
groups, will reveal a tower-like structure in successively larger 
subspaces which completely captures the structure of the exceptional 
orbit varieties.
\end{Example}
\par
\subsection*{Representations of Solvable Linear Algebraic Groups} 
\hfill\par
The most important special case for us will concern representations of 
connected solvable linear algebraic groups.  Recall that a linear algebraic 
group $G$ is {\it solvable} if there is a series of algebraic subgroups  $G = 
G_0 \supset G_1 \supset G_2 \supset \cdots G_{k-1} \supset G_k = \{ 
e\}$ with $G_{j+1}$ normal in $G_j$ such that $G_j/G_{j+1}$ is abelian 
for all $j$.  Equivalently, if $G^{(1)} = [G, G]$ is the (closed) commutator 
subgroup of G, and $G^{(j+1)}= [G^{(j)}, G^{(j)}]$, then for some $j$, $G^{(j)} 
= \{1\}$.  \par
Unlike reductive algebraic groups, representations of solvable linear 
algebraic groups need not be completely reducible.  Moreover, neither the 
representations nor the groups themselves can be classified.  Instead, the 
important property of solvable groups for us is given by the Lie-Kolchin 
Theorem (see e.g. \cite[Cor. 10.5]{Bo}), which asserts that a finite 
dimensional representation $V$ of a connected solvable linear algebraic 
group $G$ has a flag of $G$--invariant subspaces
$$  V = V_N \supset V_{N-1} \supset \cdots \supset V_1 \supset V_0 = \{
0\} \, , $$
where $\dim V_j = j$ for all $j$. We shall be concerned with nontrivial 
block representations for the actions of connected solvable linear 
algebraic groups where the $W_j$ form a special subset of a flag of 
$G$--invariant subspaces.  Then, not only will we give the block
representation, but we shall see that the diagonal blocks $D_j$ will be
given very naturally in terms of certain submatrices.  These will be
examined in \S\S\, \ref{S:CholFact}, \ref{S:ModCholreps}, 
\ref{S:NonlinLie}, 
and \ref{S:RestrBlocReps}.

\section{Operations on Block Representations}
\label{S:OpBLkReps}
\par
We next give several propositions which describe how block
representations behave under basic operations on representations.  These
will concern taking quotient representations, restrictions to
subrepresentations and subgroups, and extensions of representations.  We 
will give an immediate application of the extension property Proposition 
\ref{prop5.3} in the next section.  We will also apply the restriction and 
extension properties in  \S \ref{S:RestrBlocReps} to obtain auxiliary 
block representations which will be needed to carry out calculations in 
Part II.
\par
Let  $\rho :G \to \GL(V)$ be a block representation with decomposition
$$   V = W_k \supset W_{k-1} \supset \cdots  \supset W_{1}  \supset
W_{0} = (0)  $$
and normal algebraic subgroups
$$   G = K_k \supset K_{k-1} \supset \cdots  \supset K_{1}  \supset
K_{0} \, ,$$
with $K_j = \ker(\rho_j : G \to \GL(V/W_j))$ and $\dim K_{j} = \dim
W_{j}$, so $K_{0}$ is a finite group.  We also let
$p_j\in\cO_{V/W_{j-1}}$ be the relative
coefficient determinant for the action of $K_j/K_{j-1}$ on $W_j/W_{j-
1}$ in $V/W_{j-1}$.  \par
We first consider the induced quotient representation of $G/K_{\ell}$ on
$V/W_{\ell}$.

\begin{Proposition}[Quotient Property]
\label{prop5.1}
For the  block representation $\rho :G \to \GL(V)$ with its 
decomposition as above, the induced quotient representation $G/K_{\ell}
\to \GL(V/W_{\ell})$ is a block representation with decomposition
\begin{align*}
\overline{V}=V/W_{\ell} \, &= \, \overline{W}_{k-\ell} \supset
\overline{W}_{k-\ell-1} \supset \cdots
\supset \overline{W}_{1}  \supset \overline{W}_{0} = (0)  \qquad 
\makebox{ and} \\
\overline{G} \, = \, G/K_{\ell}\, &= \, \overline{K}_{k-\ell} \supset
\overline{K}_{k-\ell - 1} 
\supset \cdots \supset \overline{K}_{1}  \supset \overline{K}_{0}
\end{align*}
where $\overline{W}_{j} = W_{j+\ell}/W_{\ell}$ and $\overline{K}_{j} =
K_{j+\ell}/K_{\ell}$.  Then, the coefficient determinant is given by 
$\prod_{i = \ell + 1}^{k} p_i$.  \par
If $\rho$ is only a nonreduced block representation then the quotient
representation is a (possibly) nonreduced block representation.
\end{Proposition}
\begin{proof}
Let $\overline{p}_j\in\cO_{\overline{V}}$ be the relative coefficient
determinant of the representation
$\overline{K}_j/\overline{K}_{j-1}\to\GL(\overline{V}/
\overline{W}_{j-1})$
for the invariant subspace $\overline{W}_j/\overline{W}_{j-1}$.
By the basic isomorphism theorems, this representation is isomorphic to 
the representation $K_{j+\ell}/K_{j+\ell-1}\to\GL(V/W_{j+\ell-1})$ with 
the invariant subspace $W_{j+\ell}/W_{j+\ell-1}$, which has relative 
coefficient determinant $p_{j+\ell}\in\cO_V$.
Each of the polynomials $\overline{p}_j$ and $p_j$ are pullbacks of 
polynomials on the respective isomorphic spaces 
$\overline{V}/\overline{W}_{j-1}$ and $V/W_{j+\ell-1}$. 
As relative coefficient determinants are well-defined by the 
representation and invariant subspace up to multiplication by a unit, these
polynomials agree via the isomorphism $\overline{V}/\overline{W}_{j-1} 
\simeq V/W_{j+\ell-1}$ up to multiplication by a unit.  \par
By hypothesis, then,
the relative coefficient determinants
for the blocks in the quotient representation are reduced and relatively
prime.  Hence, the quotient representation is a block representation. \par
If the relative coefficient determinants for $\rho$ are not necessarily
reduced or relatively prime, then neither need be those for the quotient
representation.
\end{proof}
\par
The second operation is that of restricting to an invariant subspace and
subgroup.
\begin{Proposition}[Restriction Property]
\label{prop5.2}
Let $\rho :G \to \GL(V)$ be a block representation with its 
decomposition as above, and let $K$ be a connected linear 
algebraic subgroup with $K_{\ell} \supset K \supset K_{\ell-1}$.  Suppose 
that $W$ is a $K$--invariant subspace with $W_{\ell} \supset W \supset 
W_{\ell-1}$ and $\dim K = \dim W$.  Suppose that the coefficient 
determinant $p$ of $K/K_{\ell-1}$ on $W/W_{\ell-1}$ together with the 
restrictions of the relative coefficient determinants $p_{j\, |W}$ for the 
actions of  $K_j/K_{j-1}$ on the subspace $W_j/W_{j-1}$ in $V/W_{j-1}$ 
for $j = 1, \dots , \ell - 1$ are reduced,
and relatively prime.  Then, the restricted representation  
$\overline{\rho} : K
\to \GL(W)$ is a block representation with decomposition
\begin{align*}
  W \, &= \, \overline{W}_{\ell} \supset \overline{W}_{\ell - 1} \supset 
\cdots  
\supset \overline{W}_{1}  \supset \overline{W}_{0} = (0)  \qquad 
\makebox{ and} \\
  K \,  &= \, \overline{K}_{\ell} \supset \overline{K}_{\ell - 1} \supset 
\cdots  \supset
\overline{K}_{1}  \supset \overline{K}_{0} 
\end{align*}
where for $0 \leq j < \ell$, $\overline{W}_{j} = W_{j}$  and 
$\overline{K}_{j}$ contains $K_{j}$ as an open subgroup.
\end{Proposition}
\begin{proof}
We have given the subspaces and subgroups in the statement of the 
proposition where for $j < \ell$
\begin{equation}
\label{Eqn3.12}  
\overline{K}_j  \, = \,  \ker (K  \to \GL(W/ \overline{W}_{j})) \, = \, \ker 
(K \to  \GL(W/ W_{j})) .  
\end{equation}  
To prove that this gives a block representation, it is sufficient to show 
that $K_{j}$ is an open subgroup of $\overline{K}_{j}$ for each $j$.  It then 
follows first that $\dim(\overline{K}_j)=\dim(K_j)=\dim(W_j)$ and that 
the Lie algebras of $\overline{\k}_{j}$ and  $\k_{j}$ of $\overline{K}_{j}$, 
resp. $K_{j}$, agree.  Also, by assumption 
$\dim(\overline{K}_\ell)=\dim(W_\ell)$.  Then, the relative coefficient 
determinant of $K_j/K_{j-1}$ on the subspace $W_j/W_{j-1}$ 
in $W/W_{j-1}$ is $p_j|_W$, where
$p_j$ is the relative coefficient determinant of $K_j/K_{j-1}$
on the subspace $W_j/W_{j-1}$ in $V/W_{j-1}$.  
The conclusion of the proposition will then follow since, by assumption, 
the relative coefficient determinants are all reduced and relatively prime 
in $\cO_W$.
\par
Finally we prove that $K_{j}$ is an open subgroup of $\overline{K}_{j}$. 
Suppose not, so $\dim(K_j)<\dim(\overline{K}_j)$.  By \eqref{Eqn3.12} and 
ii) of Lemma \ref{Lem4.1}, if $X \in \overline{\k}_{j}$, then a 
representation of $\xi_X$ will have zero cofficients for the basis of 
$W/W_{j}$.  If we then compute the relative coefficient matrix 
for the action of $K/K_{j}$ on $W/W_j$, by including a 
$X \in \overline{\k}_{j}\backslash \k_{j}$ in a complementary basis to 
$\k_{j}$ in $\k$ (the Lie algebra of $K$), then the relative coefficient 
matrix would have a column identically zero, and so the relative 
coefficient determinant would be $0$.  This contradicts it being equal to 
the product of nonzero relative coefficient determinants appearing in the 
statement.  Thus, $\k_j=\overline{\k}_j$, and the statement follows.  
\end{proof}
\par
Third, we have the following proposition which allows for the extension
of a block representation yielding another block representation, providing 
a partial converse to Proposition \ref{prop5.1}.
\begin{Proposition}[Extension Property]
\label{prop5.3}
Let $\rho :G \to \GL(V)$ be a representation of a connected linear 
algebraic 
group, so that $W \subset V$ is a $G$--invariant subspace and
$K = \ker (G \to \GL(V/W))$ with $\dim (K) = \dim (W)$.  Suppose that the
quotient representation $\overline{\rho}: G/K \to  \GL(V/W)$ is a block
representation with decomposition
\begin{align*}
  V/W\,  &= \, \overline{W}_{\ell} \supset \overline{W}_{\ell-1} \supset 
\cdots  
\supset \overline{W}_{1}  \supset \overline{W}_{0} = (0)  \qquad 
\makebox{ and}  \\
  \overline{G} \,  = \, G/K \,  &= \, \overline{K}_{\ell} \supset 
\overline{K}_{\ell - 1} 
\supset
\cdots  \supset \overline{K}_{1}  \supset \overline{K}_{0}\, , 
\end{align*}
for which the relative coefficient determinant for the action of $K$ on
the subspace $W$ in $V$ is reduced and relatively prime to the coefficient 
determinant for $\overline{\rho}$.  Then, $\rho$ is a block representation 
with decomposition
\begin{align*}
  V \, &= \, W_{\ell+1} \supset W_{\ell} \supset \cdots  \supset W_{1}
\supset W_{0} = (0)  \qquad \makebox{ and}  \\
  G \, &= \, K_{\ell+1 } \supset  K_{\ell} \supset \cdots  \supset  K_{1}
\supset  K_{0} = \{ Id\} \, .
\end{align*}
Here $W_{1} = W$, $K_1 = K$, and for $j = 1, \dots, \ell$, $W_{j+1} =
\pi^{-1}(\overline{W}_j)$ and $K_{j+1 } = \pi^{\prime \,
-1}(\overline{K}_j)$ for $\pi :
V \to V/W$ and $\pi^{\prime} : G \to G/K$ the projections.
\par
If instead $\overline{\rho}$ has a nonreduced block structure or the 
relative
coefficient determinant for the action of $K$ on $W$ is nonreduced or not
relatively prime to the coefficient determinant for $\overline{\rho}$, 
then,
$\rho$ is a nonreduced block representation.
\end{Proposition}
\par
\begin{proof}
Again the proposition gives the form of the decomposition, provided
we verify the properties.  By our assumptions, $\dim K_{j} = \dim W_{j}$ 
for all
$j$.  For $1 \leq j \leq \ell$, with $\pi^{\prime} : G \to G/K$ as above,
$$  \ker (G \to \GL(V/W_j)) = \pi^{\prime \, -1}(\ker (G/K \to
\GL(\overline{W}_\ell/
\overline{W}_{j-1})))  = \pi^{\prime \, -1}(\overline{K}_{j-1}) =  K_j\, . $$  
\par
Finally, using the stated decomposition, the coefficient matrix has a
lower block triangular form.  Then, the coefficient determinant for the
representation of $\rho :G \to \GL(V)$ is the nonzero product of the 
relative coefficient determinants, which equals the product of the 
relative coefficient determinant of $K$ acting on $W$ and the coefficient
determinant of $G/K$ acting on $V/W$ (pulled back to $V$).  Hence
it is a block representation.
\end{proof}

\begin{Remark}
If we extend a block representation as in Proposition
\ref{prop5.3}
and then form the quotient by $W$ using Proposition 
\ref{prop5.1}, 
we recover the original block representation.
\end{Remark}

\section{Towers of Linear Algebraic Groups and Representations}
\label{S:TowReps}
\par
The two key questions concerning block representations are: \par
\begin{itemize}
\item[i)]  How do we find the $G$-invariant subspaces $W_j$?
\item[ii)] Given the $\{ W_j\}$, what specifically are the diagonal blocks
$D_j$?
\end{itemize}
The first question becomes more approachable when we
have a series of groups with a corresponding series of representations.
\begin{Definition}
A {\em tower of linear algebraic groups} $\bG$ is a sequence of such
groups \par
$$ \{ e\} = G_0 \subset G_1 \subset G_2 \subset \cdots \subset G_k
\subset \cdots  \, .$$
Such a tower has a {\em tower of representations} $\bV = \{ V_j\}$  if
$$ (0) = V_0 \subset V_1 \subset V_2 \subset \cdots \subset V_k \subset
\cdots $$
where each $V_k$ is a representation of $G_k$, and for the inclusion maps
$\iti_k : G_{k} \hookrightarrow G_{k+1}$, and $\itj_k : V_{k}
\hookrightarrow V_{k+1}$, the mapping $(\iti_k, \itj_k): (G_{k}, V_{k}) \to
(G_{k+1}, V_{k+1})$ is a homomorphism of groups and representations.

\end{Definition}
Then, we identify within towers when the block representation structures
are related.
\begin{Definition}
\label{Def6.2}
 A tower of connected linear algebraic groups and representations $(\bG, 
\bV)$ has a {\em block
structure} if: for all $\ell \geq 0$ the following hold:
\begin{itemize}
\item[i)]  Each $V_{\ell}$ is a block representation of $G_{\ell}$ via the
decompositions
$$  G_{\ell} = K_k^{\ell} \supset K_{k-1}^{\ell} \supset \cdots  \supset
K_{1}^{\ell}  \supset K_{0}^{\ell}  $$
where $K_{0}^{\ell}$ is a finite group, and
$$   V_{\ell} = W_k^{\ell} \supset W_{k-1}^{\ell} \supset \cdots  \supset
W_{1}^{\ell} \supset W_{0}^{\ell} = (0) . $$

\item[ii)]  For each $\ell > 0$ the composition of the natural
homomorphisms of representations
$$  (G_{\ell-1}, V_{\ell-1}) \to (G_{\ell}, V_{\ell})  \to
(G_{\ell}/K_{1}^{\ell}, V_{\ell}/W_{1}^{\ell})  $$
is an isomorphism of representations.
\end{itemize}\par
If instead in i) we only have nonreduced block representations, then we
say that the tower has a {\em nonreduced block structure}.
\end{Definition}
\par
In particular, for all $\ell$ the sequence
$$\begin{CD}
0 @>>> K_1^{\ell} @>>> G_\ell @>>> G_\ell/K_1^{\ell} @>>> 0
\end{CD}$$
splits, as $G_\ell/K_1^{\ell}\cong i_{\ell-1}(G_{\ell-1})\subset
G_{\ell}$. 
Moreover, the representations of these groups and their collections of 
invariant subspaces are compatible.  This will also allow us to use the 
properties of \S \ref{S:OpBLkReps} to give inductive criteria at each stage 
to establish the tower block structure.  \par  
We first deduce an important consequence for
the collection of exceptional orbit varieties.  Specifically the tower 
structure will allow us to canonically decompose the exceptional orbit 
varieties for the representations in terms of those from lower dimensions 
together with the generalized determinant varieties.  \par
Then, for such a tower of representations with a block structure (or
nonreduced block structure) we have the following basic theorem which
summarizes the key consequences and will yield the results for many 
spaces of matrices.  
\begin{Thm}
\label{Towerthm}
Suppose $(\bG, \bV)$ is a tower of connected linear algebraic groups and 
representations which has a block structure.   Let $\cE_{\ell}$ be the 
exceptional orbit variety for the action of $G_{\ell}$ on $V_{\ell}$.  
\begin{itemize}
\item[i)]  For each ${\ell}$, $\cE_{\ell}$ is a linear free divisor.
\item[ii)]  The quotient space $V_{\ell}/W_{1}^{\ell}$ can be naturally
identified with $V_{\ell - 1}$ as $G_{\ell-1}$--representations.
\item[iii)]  The generalized determinant variety $\cD_{\ell}$ for the 
action of
$G_{\ell}$ on $V_{\ell}$ satisfies $\cE_{\ell} = \cD_{\ell} \cup 
\pi_{\ell}^*\cE_{{\ell}-1}$, where
$\pi_{\ell}$ denotes the projection $V_{\ell} \to V_{{\ell}-1}$ induced 
from ii).
\end{itemize}
\par
If instead $(\bG, \bV)$ has a nonreduced block structure, then each
$\cE_j$ is a linear free* divisor.
\end{Thm}
\begin{proof}
First, it is immediate from Theorem \ref{BlkRepFrDiv} that each 
$\cE_{\ell}$ is a linear free divisor.  Furthermore, by property ii) for a 
block structure for towers, the composition $V_{\ell-1} \to V_{\ell} \to 
V_{\ell}/W_{1}^{\ell}$ is an isomorphism which defines for each $\ell$ a 
projection $\pi_{\ell} : V_{\ell} \to  V_{\ell-1}$ with kernel 
$W_{1}^{\ell}$ which is equivariant for the induced $G_{\ell-1}$-action. 
This establishes ii).  
\par
To show iii), note that by property ii) we may view
$(G_\ell,V_\ell)$ as an extension of $(G_{\ell-1},V_{\ell-1})$
in the sense of 
Proposition \ref{prop5.3}.  By the proof of Proposition \ref{prop5.3},
$\cE_\ell$ is the union of $\pi_\ell^*(\cE_{\ell-1})$ and the locus
defined by the relative
coefficient determinant for the action of $K_1^{\ell}$ on the subspace
$W_1^\ell$ in $V_\ell$, i.e., the generalized determinant variety $D_\ell$.
\end{proof}
\par
We can also give a levelwise criterion that a tower have a block
structure.
\begin{Proposition}
\label{Prop7.2}
Suppose that a tower of linear algebraic groups and representations 
$(\bG, \bV)$  satisfies the following conditions: the representation of 
$G_1$ on $V_1$ is a block representation and for all $\ell \geq 1$ the 
following hold:
\begin{itemize}
\item[i)]  The representation $V_{\ell}$ of $G_{\ell}$ is an
equidimensional representation and has an invariant subspace $W_{\ell}
\subset V_{\ell}$ of the same dimension as $K_{\ell}$, the connected
component of the identity of $\ker(G_{\ell} \to \GL(V_{\ell}/W_{\ell}))$.
\item[ii)] The action of $K_{\ell}$ on $W_{\ell}$ has a relatively open
orbit in $V_{\ell}$, and the relative coefficient determinant for $K_{\ell}$
on $W_{\ell}$ in $V_{\ell}$ is reduced and relatively prime to the
coefficient determinant of $G_{\ell-1}$ acting on $V_{\ell-1}$ (pulled
back to $V_{\ell}$ via projection along $W_{\ell}$).
\item[iii)]  The composition of the natural homomorphisms of
representations
$$  (G_{\ell-1}, V_{\ell-1}) \to (G_{\ell}, V_{\ell})  \to (G_{\ell}/K_{\ell},
V_{\ell}/W_{\ell})  $$
is an isomorphism of representations.
\end{itemize} \par
Then, the tower $(\bG, \bV)$ has a natural block structure (with the 
decomposition for $(G_{\ell}, V_{\ell})$ given by
(\ref{Eqn7.1a}) and (\ref{Eqn7.1b}) below).  \par
If the representation of $G_1$ on $V_1$ only has a nonreduced block
representation or in condition ii) the relative coefficient determinants 
are not all reduced or not relatively prime then the tower has a 
nonreduced block
structure.
\end{Proposition}
\begin{Remark}
 If $p_j$ denotes the relative coefficient determinant for the action of
$K_j$ on $W_j$ in $V_j$, then it will follow by the proof of Theorem
\ref{Towerthm} that it is sufficient that $p_{\ell}$ is reduced and
relatively prime to each of the $p_j$, $j = 1, \dots , \ell - 1$  (pulled back
to $V_{\ell}$ by the projection $\pi_{i} : V_{\ell} \to V_{i}$ consisting of
the compositions of projections along the $W_j$).
\end{Remark}
\begin{proof}
We shall show by induction on $\ell$ that each representation of
$G_{\ell}$ on $V_{\ell}$ is a block representation.
Each inductive step is an application of Proposition \ref{prop5.3}.
We begin by defining the decomposition for $(G_{\ell}, V_{\ell})$.  \par
We first suppose that we have a trivial block
decomposition (i.e. single block) for $G_1$ on $V_1$.  We let $\pi_j : G_j 
\to G_{j-1}$ denote the projection obtained from the
composition of the projection $ G_j \to G_{j}/K_{j}$ with the inverse of
the isomorphism given by condition iii).  We can analogously define
$\pi_j^{\prime} : V_j \to V_{j-1}$.  Composing successively the $\pi_j$
we obtain projections $\pi_{\ell}^j : G_{\ell} \to G_{j}$. Likewise we
define $\pi_{\ell}^{j\, \prime} : V_{\ell} \to V_{j}$ by successive
compositions of the $\pi_j^{\prime}$.  \par
Then, we define for $1 < j \leq \ell$,
\begin{equation}
\label{Eqn7.0a}
W_{j}^{\ell} = \pi_{\ell}^{j\, \prime\, -1}(W_j)\quad \makebox{ and }
\quad  K_{j}^{\ell} = \pi_{\ell}^{j\, -1}(K_j) .
\end{equation}
For $j = 1$, we let $W_{1}^{\ell} = W_{\ell}$ and $K_{1}^{\ell} = K_{\ell}$
(also $K_{0}^{\ell} = \ker(G_{\ell} \to \GL(V_{\ell}))$).  \par
Then, the decomposition is given by
\begin{equation}
\label{Eqn7.1a}
   V_{\ell}  =  W_{\ell}^{\ell} \supset W_{\ell-1}^{\ell} \supset \cdots
\supset W_{1}^{\ell} \supset W_{0}^{\ell} = (0) \qquad \makebox{ and}
\end{equation}
\begin{equation}
\label{Eqn7.1b}
 G_{\ell} = K_{\ell}^{\ell} \supset K_{\ell-1}^{\ell} \supset \cdots
\supset K_{1}^{\ell}  \supset K_{0}^{\ell} \, \, .
\end{equation}
\par

Then, for $\ell = 1$, the decomposition given by (\ref{Eqn7.1a}) and
(\ref{Eqn7.1b}) is that for $G_1$ on $V_1$.  We assume it is true for all
$j < \ell$, and consider the  representation of $G_{\ell}$ on $V_{\ell}$.
\par
By assumption, $G_{\ell}/K_{\ell} \simeq G_{\ell-1}$ and
$V_{\ell}/W_{\ell} \simeq V_{\ell-1}$ as $G_{\ell-1}$ representations.
By the assumption, the relative coefficient determinant for the
representation of $K_{\ell}$ on $W_{\ell}$ is reduced and relatively prime
to the coefficient determinant of $G_{\ell - 1}$ acting on $V_{\ell-1}$.
Hence, we may apply Proposition \ref{prop5.3} to conclude that the
representation of $G_{\ell}$ on $V_{\ell}$ has a block representation
obtained by pulling back that of $G_{\ell-1}$ on $V_{\ell-1}$ via the
projections $\pi_{\ell} : G_{\ell} \to G_{\ell-1}$ and $\pi_{\ell}^{\prime} :
V_{\ell} \to V_{\ell-1}$. Specifically, for $j > 1$ we let 
\begin{equation}
\label{Eqn7.2b}
W_{j}^{\ell} = \pi_{\ell}^{\prime\, -1}(W_{j}^{\ell-1}) \quad \makebox{ and 
}
\quad  K_{j}^{\ell} = \pi_{\ell}^{-1}(K_{j}^{\ell-1})\,\, .
\end{equation}
For $j =1$, $W_{1}^{\ell} = W_{\ell}$ and $K_{1}^{\ell} = K_{\ell}$,
while $K_{0}^{\ell} = \ker(G_{\ell} \to \GL(V_{\ell}))$ is a finite group.
However, by the inductive assumption, (\ref{Eqn7.2b}) gives exactly
$W_{j}^{\ell}$ and $K_{j}^{\ell}$ defined for (\ref{Eqn7.0a}).  This
establishes the inductive step.  Then, assumption iii) establishes the 
second condition for the tower
having a block structure.  \par
If instead of having a trivial block decomposition for $G_1$ on $V_1$; we 
have a full block representation for $G_1$ on $V_1$, involving say $N$ 
blocks for $(G_1; V_1)$, then we can refine the block representation
given here for $(G_{\ell}; V_{\ell})$ by pulling the block decomposition for 
$(G_1; V_1)$ back via the $\pi_{\ell}^1$ and $\pi_{\ell}^{1\,\prime}$ to 
obtain a block representation with $N + \ell - 1$ blocks.  \par
If $(G_1, V_1)$ only has a nonreduced block structure or the relative
coefficient determinants are not reduced or not relatively prime, then the
above proof only shows the $(G_{\ell}, V_{\ell})$ have nonreduced block
structures.
\end{proof}
\par
The use of this Proposition to establish that certain towers of
representations have block structure will ultimately require that we
establish that the relative coefficient determinants are irreducible and
relatively prime.  The following Lemma will be applied in later sections 
for each of the families that we consider.
\begin{Lemma}
\label{Lem7.4}
Suppose $f\in\C[x_1 , \ldots , x_n, y]$,
and $g = \frac{\partial f}{\partial y} \in\C[x_1, \dots , x_n]$.
\begin{itemize}
\item[i)]  If $\gcd(f, g) = 1$ then  $f$ is irreducible.
\item[ii)]  If for each irreducible factor $g_1$ of $g$, there is a $(x_{1 0},
\dots , x_{n 0}, y_0)$ so that $g_1(x_{1 0}, \dots , x_{n 0}) = 0$  while
$f(x_{1 0}, \dots , x_{n 0}, y_0) \neq 0$, then $f$ is irreducible.
\end{itemize}
\end{Lemma}
\begin{proof}[Proof of Lemma \ref{Lem7.4}] \par
i) is a consequence of the Gauss lemma applied to the polynomial ring 
$R[y] = \C[x_1, \dots , x_n, y]$ where $R = \C[x_1, \dots , x_n]$.  By the 
hypothesis, $f = g\cdot y + g_0$ has degree $1$ in $y$ with $g_0, g \in R$.  
Then, $\text{content}(f) = 1$ provided $\gcd(g_0, g) (= \gcd(f, g)) = 1$.  
The assumptions in ii) imply $\gcd(f, g) = 1$.
\end{proof}

\section{Basic Matrix Computations for Block Representations}
\label{S:MatrComps}
\par
To apply the results of the preceding sections, we must first perform
several basic calculations for two basic families of representations.  
While the calculations themselves are classical and straightforward,
we collect them 
together  in a form immediately applicable to the towers of 
representations we consider.  We let $M_{m, p}$ denote the space of 
$m \times p$ complex matrices.  We consider the following 
representations.
\flushpar
\begin{itemize}
\item[i)] the {\it linear transformation representation} on $M_{m, p}$: 
defined by
\begin{align}
\label{Eqn8.1a}
\psi : \GL_m(\C) \times \GL_p(\C) &\to \GL(M_{m, p})  \\
\psi(B, C)(A) \,\, &= \,\, B\, A\, C^{-1}  \notag
\end{align}

\item[ii)] the {\it bilinear form representation} on $M_{m, m}$: 
defined by
\begin{align}
\label{Eqn8.1b}
\theta : \GL_m(\C)  &\to \GL(M_{m, m})  \\
\theta(B)(A) \,\, &= \,\, B\, A\, B^T . \notag
\end{align}
\end{itemize}
\par
We will then further apply these computations to the restrictions to
families of solvable subgroups and subspaces which form towers
$\rho_{\ell} : G_{\ell} \to \GL(V_{\ell})$ of representations.  For these
representations and their restrictions, we will carry out the following.
\par
\begin{enumerate}
 \item identify a flag of invariant subspaces $\{ V_j\}$;
\item  from among the invariant subspaces, identify distinguished
subspaces  $W_j$ and the corresponding normal subgroups $K_j = \ker( G
\to \GL(V/W_j))$;
\item  compute the representation vector fields for a basis of the Lie
algebra;  and
\item compute the relative coefficient matrix for the representation of
$K_j/K_{j-1}$ on $W_j/W_{j-1}$ in $V/W_{j-1}$ using special bases for
the Lie algebra $\bk_j/ \bk_{j-1}$ (the Lie algebra of $K_j/K_{j-1}$) and
$W_j/W_{j-1}$ to determine the  diagonal blocks in the block
representation.
\end{enumerate}
\par

\subsection*{Linear Transformation Representations} \hfill
\par
Next, we let $B_m$ denote the Borel subgroup of $\GL(\C^m)$ consisting 
of
invertible lower triangular matrices, and $B_p^T$ denote the subgroup of
$\GL(\C^p)$ consisting of invertible upper triangular matrices (this is the
transpose of $B_p$).  We consider the representation $\rho$ of $B_m
\times B_p^T$ on $M_{m, p}$ obtained by restricting the linear
transformation  representation $\psi$ .
Eventually we will be interested in the cases $p = m$ or $m+1$.
\par
\vspace{1ex} \flushpar
\subsubsection*{Invariant Subspaces and Kernels of Quotient
Representations}  \hfill
\par
To simplify notation, for fixed $m$ and $p$ we denote $M_{m, p}$ as $M$.
We first define for given $0 \leq \ell \leq m$ and  $0 \leq k \leq p$ the
subspace $M^{(\ell , k)}$ of $M$ which consists of matrices for which the
upper left-hand $(m -\ell) \times (p -k)$ submatrix is $0$.  Thus, $\dim
M^{(\ell , k)}$ decreases with decreasing $\ell$ and $k$.  Given $m$ and
$p$ we let $E_{i, j}$ denote the elementary $m \times p$ matrix with $1$
in the $i,j$--th position, and $0$ elsewhere. \par
We first observe
\begin{Lemma}
\label{Lem8.2}
The subspaces $M^{(\ell , k)}$ are invariant subspaces for the
representation of $B_m\times B_p^T$.
\end{Lemma}
\begin{proof}
We partition $m$ into $m - \ell$ and $\ell$ and $p$ into $p-k$ and $k$, and
write our matrices in block forms with the rows and columns so
partitioned.  Then,
\begin{equation}
\label{Eqn8.3}  \begin{pmatrix}
B^{\prime} &0 \\
*  & *
\end{pmatrix}\cdot \begin{pmatrix}
A^{\prime}  &* \\
*  & *
\end{pmatrix}\cdot \begin{pmatrix}
C^{\prime\, -1} &* \\
0  & *
\end{pmatrix} \quad = \quad   \begin{pmatrix}
B^{\prime}\,A^{\prime}\,C^{\prime\, -1}  &* \\
*  & *
\end{pmatrix} \,\, .
\end{equation}
Then, (\ref{Eqn8.3}) shows that if $A^{\prime} = 0$ then so is
$B^{\prime}\,A^{\prime}\,C^{\prime\, -1}$.
\end{proof}
Then, we obtain an induced quotient representation
$$\rho_{\ell, k} : B_m \times B_p^T \to\GL(M/M^{(\ell, k)}) \, .$$
We consider the subgroup $K^{(\ell, k)}$ consisting of elements of
$B_m \times B_p^T$ of the form
\begin{equation}
\label{Eqn8.4}
\left(
\begin{pmatrix} \lambda \cdot I_{m - \ell} & 0 \\ * & * \end{pmatrix},
\begin{pmatrix} \lambda \cdot I_{p - k} & * \\ 0 & * \end{pmatrix}
\right),
\qquad \lambda\in\C^* \, .
\end{equation}
This subgroup has the following role.
\begin{Lemma}
\label{Lem8.4}
For the quotient representation $\rho_{\ell, k}$,  $\ker (\rho_{\ell, k}) =
K^{(\ell, k)}$.
\end{Lemma}
\begin{proof}[Proof of Lemma \ref{Lem8.4}]
We use the partition as in equation (\ref{Eqn8.3}).  The product is in $\ker
(\rho_{\ell, k})$ if and only if
\begin{equation}
\label{Eqn8.5}
   B^{\prime}\,A^{\prime}\,C^{\prime\, -1}  \,\,=  \,\, A^{\prime}
\end{equation}
 for all $(m-\ell) \times (p-k)$ matrices $A^{\prime}$.  It follows that
$K^{(\ell, k)} \subseteq \ker (\rho_{\ell, k})$.  \par
For the reverse inclusion, we let $B^{\prime} = (b_{i, j})$ and $C^{\prime} =
(c_{i, j})$ and examine (\ref{Eqn8.5}) for $A^{\prime} = E_{i, j}$, the
$(m-\ell) \times (p-k)$--elementary matrices for $1 \leq i \leq m-\ell$,
and $1 \leq j \leq p-k$.  We see that $b_{i, j} = 0$ and $c_{i, j} = 0$ for
$i \neq j$, and then $b_{i, i} = b_{j, j}$ and $c_{i, i} = c_{j, j}$ for all $i$
and $j$.  This implies $B^{\prime} = \lambda I_{m-\ell}$, $C^{\prime} =
\kappa I_{p-k}$, and (\ref{Eqn8.3}) implies $\lambda = \kappa$.
\end{proof}
\par
We note that a consequence of Lemma \ref{Lem8.4}, is that the
representation $\rho$ is not faithful, and hence cannot be an
equidimensional representation.  We shall see in the next section that by
restricting to appropriate solvable subgroups we can overcome this in
different ways.   First, we determine the associated representation vector
fields. \par
\vspace{1ex} \flushpar
\subsubsection*{Representation Vector Fields}  \hfill
\par
The derivative of $\rho$ at $(I_m, I_p)$ is given by straightforward
calculation to be
\begin{equation}
\label{Eqn8.7}
    d\rho(B, C)(A) \,\, = \,\, B\, A -  A\, C
\end{equation}
for $(B, C) \in \gl_m \oplus \gl_p$ and $A \in M$.  This computes $\pd{
}{t}(\exp(tB)\, A\, \exp(tC)^{-1})_{|t = 0}$, and hence is the representation
vector field corresponding to $(B, C) $ evaluated at $A$.  We obtain two
sets of vector fields
\begin{equation}
\label{Eqn8.8}
\xi_{i, j} = \xi_{(E_{i, j}, 0)}  \quad \makebox{ and } \quad  \zeta_{i, j} =
\xi_{(0, E_{i, j})}\, \, .
\end{equation}
We calculate them using $(\ref{Eqn8.7})$ to obtain for $A = (a_{i,j})$, 
\begin{align}
\label{Eqn8.9}
\xi_{k, \ell}(A) \,\, &= \,\, E_{k, \ell}\, A \,\, = \,\, \sum_{s = 1}^{p} 
a_{\ell,s} E_{k, s} \quad \makebox{ and } \\
\quad  \zeta_{k, \ell}(A) \,\, &= \,\,  - A\,E_{k, \ell} \,\, = \,\, - \sum_{s
= 1}^{m} a_{s, k} E_{s, \ell}\,\, . \notag
\end{align}
These can be described as follows: $\xi_{k, \ell}$
associates to the matrix $A$ the matrix  all of whose rows are zero
except for the $k$--th which is the $\ell$--row of $A$.  Similarly  
$\zeta_{k, \ell}$ associates to the matrix $A$ the matrix  all
of whose columns are zero except for the $k$--th column which is minus
the $\ell$--th column of $A$.

\subsection*{Bilinear Form Representations} \hfill
\par
We next make analogous computations for the bilinear form
representations.
 \par
\vspace{1ex} \flushpar
\subsubsection*{Invariant Subspaces and Kernels of Quotient
Representations}  \hfill

For the bilinear form representation $\theta$ on $M = M_{m, m}$, we
observe that it is obtained by composition of $\rho$ (for the case $p = m$)
with the Lie group homomorphism $\gs : B_m \to B_m \times B_m^T$
defined by $\gs(B) = (B, (B^{-1})^T)$.  Since $\theta = \rho\circ \gs$, it is
immediate that the invariant subspaces $M^{(\ell , k)}$ for $B_m \times
B_m^T$ via $\rho$ are also invariant for $B_m$ via $\theta$.  Also, it
immediately follows that for the quotient representation
$$  \theta_{\ell, k} : B_m \to \GL(M/M^{(\ell , k)})\, , $$
$\ker(\theta_{\ell, k}) = \gs^{-1}(\ker(\rho_{\ell, k}))$.  However, by
Lemma \ref{Lem8.4}, $\ker(\rho_{\ell, k}) = K^{(\ell, k)}$.  Thus, an 
element $B \in\ker(\theta_{\ell, k})  = \gs^{-1}(K^{(\ell, k)})$ has the form
\begin{equation}
\label{Eqn8.9b}
B  \,\, = \,\,  \begin{pmatrix} \lambda \cdot I_{\ell} & 0 \\ * & *
\end{pmatrix},
\qquad \lambda\in\C^* \,\, .
\end{equation}
Also, by (\ref{Eqn8.4})
\begin{equation}
\label{Eqn8.10}
(B^{-1})^T \,\, = \,\,  \begin{pmatrix} \lambda^{-1} \cdot I_{\ell} & * \\ 0
& * \end{pmatrix}  \,\, = \,\, \begin{pmatrix} \lambda \cdot I_{k} & * \\ 0
& * \end{pmatrix} \,\, .
\end{equation}
Hence, $\lambda = \pm 1$, and
\begin{equation}
\label{Eqn8.11}
B  \,\, = \,\,  \begin{pmatrix} \pm I_{r} & 0 \\ * & * \end{pmatrix},
\qquad \makebox{ where } r = \max\{\ell, k\}\, .
\end{equation}
We summarize this in the following Lemma.  
\begin{Lemma}
\label{Lem8.12}
For the bilinear form representations,
\begin{enumerate}
\item The $M^{(\ell , k)}$ are invariant subspaces.
\item  The kernel of the quotient representation $\theta_{\ell, k}$
consists of the elements of the form (\ref{Eqn8.11}).
\end{enumerate}
\end{Lemma}
\par
 \par
\vspace{1ex} \flushpar
\subsubsection*{Representation Vector Fields}  \hfill
\par
We can compute the representation vector fields either by using the
naturality of the exponential diagram or by directly computing $d\theta$.
In the first case, we see that corresponding to $E_{k, \ell}$ is the vector
field $\xi_{E_{k, \ell}}  =  \xi_{k, \ell} - \zeta_{\ell, k}$ using the
notation of (\ref{Eqn8.8}).  \par
Alternatively, the corresponding representation for Lie algebras $\bb_m$
sends $B \in \bb_m$ to the linear transformation sending $A \mapsto B\,
A + A\, B^T$.  This also defines the corresponding representation vector
field $\xi_B$ at $A$.  Applied to $E_{k, \ell}$, we obtain
\begin{equation}
\label{Eqn8.13}
\xi_{E_{k, \ell}}(A) \,\, = \,\,  E_{k, \ell}\, A \, + \, A\, E_{\ell, k}\,\, .
\end{equation}
\par
This action can be viewed as the action on bilinear forms defined by
matrices $A$.  We will eventually restrict this action to symmetric and
skew--symmetric bilinear forms. We apply the above analysis to this
representation.  \par
To continue further, we next identify the solvable subgroups to which we 
will
restrict the representations in order to obtain equidimensional
representations.

\par
\section{Cholesky-Type Factorizations as Block Representations of
Solvable Linear Algebraic Groups}
\label{S:CholFact}
\par
In this section, we explain how the various forms of classical 
``Cholesky-type factorization'' can be understood via representations 
of solvable groups on spaces of matrices leading to the construction of 
free (or free*) divisors containing the variety of singular matrices.  \par
Traditionally, it is well--known that certain matrices can be put in
normal forms after multiplication by appropriate matrices.  The basic
example is for symmetric matrices, where a symmetric matrix $A$ can be
diagonalized by composing it with an appropriate invertible matrix $B$ to
obtain $B\cdot A \cdot B^T$.  The choice of $B$ is highly nonunique.  For
real matrices, Cholesky factorization gives a unique choice for $B$
provided $A$ satisfies certain determinantal conditions.  
More generally, by \emph{Cholesky-type factorization} we mean a
general collection of results for factoring real matrices into products of
upper and lower triangular matrices.  These factorizations are 
traditionally used to simplify the solution of certain problems in applied 
linear algebra.  For the cases of symmetric matrices and LU decomposition 
for general $m \times m$ matrices see \cite{Dm} and for skew symmetric 
matrices see \cite{BBW}. \par
Here we state the versions of these theorems for complex matrices.  The 
complex versions can be proven either by directly adapting the real proofs, 
as in \cite{P}, or they will also follow from Theorem \ref{CholTow}.  \par
 Let $A = (a_{i j})$ denote an $m \times m$ complex matrix which
may be symmetric, general, or skew-symmetric.  We let $A^{(k)}$ denote
the $k \times k$ upper left hand corner submatrix. \par

\begin{Thm}[Complex Cholesky-Type Factorization] \hfill
\label{ComplCholFac}
\begin{enumerate}
\item  {\em Complex Cholesky factorization: } If $A$ is a
complex symmetric matrix with $\det (A^{(k)}) \neq 0$  for $k = 1, \dots,
m$, then there exists a lower triangular matrix $B$, which is unique up to
multiplication by a diagonal matrix with diagonal entries $\pm 1$, so that
$A = B \cdot B^T$.
\item {\em Complex LU factorization: } If $A$ is a general complex
matrix with $\det (A^{(k)}) \neq 0$  for $k = 1, \dots, m$, then there
exists a unique lower triangular matrix $B$ and a unique upper triangular
matrix $C$ which has diagonal entries $= 1$ so that $A = B \cdot C$.
\item  {\em Complex Skew-symmetric Cholesky factorization : } If $A$ is
a skew-symmetric matrix for $m  = 2\ell$ with $\det (A^{(2k)}) \neq 0$
for $k = 1, \dots, \ell$, then there exists a lower block triangular matrix
$B$ with $2 \times 2$--diagonal blocks of the form a)
in (\ref{Eqn9.1}) with complex entries $r$ (i.e. $ = r\cdot I$), so that $A = 
B \cdot J \cdot B^T$, for $J$
the $2\ell \times 2\ell$ skew-symmetric matrix with
$2 \times 2$--diagonal blocks of the form b) in (\ref{Eqn9.1}).  Then, $B$ 
is unique up to multiplication by block diagonal matrices with a $2 \times 
2$ diagonal blocks $= \pm I$.  For $m = 2\ell + 1$, then there is again a 
unique factorization except now $B$ has an additional entry of $1$ in the 
last diagonal position, and $J$ is replaced by $J^{\prime}$ which has $J$ 
as the upper left corner $2\ell \times 2\ell$ submatrix, with remaining 
entries $ = 0$.
\begin{equation}
\label{Eqn9.1}
a) \qquad \begin{pmatrix}
r &0 \\ 0  & r
\end{pmatrix} , \quad r > 0\qquad \makebox{ and } \quad b) \qquad  \,\,
\begin{pmatrix}
0 & -1 \\ 1  &  0
\end{pmatrix}
\end{equation}
\end{enumerate}
\end{Thm}
\par

\subsection*{Complex Cholesky Factorizations via Solvable Group
Representations}
 \par
We can view these results as really statements about representations
of solvable groups on spaces of $m \times m$ complex matrices which
will either be symmetric, general, or skew-symmetric (with $m$ even).
We consider for each of these cases the analogous representations of
solvable linear algebraic groups which we shall show form towers of
(possibly nonreduced) block representations for solvable groups.  \par
\vspace{1ex}
\flushpar
{\it  General $m \times m$ Complex Matrices : }
As earlier $M_{m, m}$ denotes the space of $m \times m$ general complex
matrices, with $B_m$ the Borel subgroup of invertible lower triangular
$m \times m$ matrices.  We also let $N_m$ be the unipotent subgroup of
$B_m^T$, consisting of the invertible upper triangular $m \times m$
matrices with $1$\rq s on the diagonal.  The representation of $B_m
\times N_m$ on $M_{m, m}$ is the restriction of the linear transformation
representation  (\ref{Eqn8.1a}).  
The inclusion homomorphisms  $B_{m-1} \times N_{m-1} \hookrightarrow
B_m \times N_{m} $ and inclusions $M_{m-1, m-1} \hookrightarrow M_{m,
m}$ are defined as in (\ref{Eqn9.2a}).
\begin{equation}
\label{Eqn9.2a}
B \mapsto \begin{pmatrix}
B &0 \\ 0  & 1
\end{pmatrix}
\qquad C \mapsto \begin{pmatrix}
C &0 \\ 0  & 1
\end{pmatrix} \qquad \makebox{ and } \qquad  A \mapsto \begin{pmatrix}
A &0 \\ 0  & 0
\end{pmatrix}
\end{equation}
These then define a tower of representations of connected solvable 
algebraic groups. \par
Second we consider restrictions of the bilinear form representations.  We
may decompose $M_{m, m}$, viewed as a representation of the Borel
subgroup $B_m$, as $M_{m, m} = \Sym_m \oplus \Sk_m$, where 
$\Sym_m$
denotes the space of $m \times m$ complex symmetric matrices and
$\Sk_m$ the space of skew--symmetric matrices.  Hence, we can restrict
the representation to each summand.  \par
\vspace{1ex} \flushpar
{\it  Complex Symmetric Matrices : }
The representation of $B_m$ on $\Sym_m$ is the restriction of the 
bilinear
form representation.  
The inclusion homomorphisms  $B_{m-1} \hookrightarrow B_m$ and
inclusions $\Sym_{m-1} \hookrightarrow \Sym_m$ are defined as in
(\ref{Eqn9.2a}) and
define a tower of solvable group representations.  \par
\vspace{1ex} \flushpar
{\it  Complex Skew-Symmetric Matrices : }
If instead we consider the representation on the summand $\Sk_m$, then
we further restrict to a subgroup of $B_m$.  For $m = 2\ell$ or 
$m = 2\ell + 1$, we let $D_m$ denote the subgroup of $B_m$ consisting of 
all lower triangle matrices of the type described in (3) of Theorem 
\ref{ComplCholFac}.  The representation of $D_m$ on $\Sk_m$ is the 
restricted representation.  
The inclusion homomorphism  $D_{m-1} \hookrightarrow D_m$ and
inclusion $\Sk_{m-1} \hookrightarrow \Sk_m$ are as in (\ref{Eqn9.2a}); 
and 
together these representations again form a tower of representations of 
connected solvable algebraic groups.
\\

The representations in each of these cases are equidimensional
representations.  Simple counting arguments show the groups and vector
spaces have the same dimension.  Moreover, in each case the subgroups
intersect the kernels of the representations $\psi$ and $\theta$ in finite
subgroups.  Hence they are equidimensional.  \par
The corresponding Cholesky-type factorization then asserts that the
representation has an open orbit and that the exceptional orbit variety is
defined by the vanishing of one of the conditions for the existence of the
factorization.  The open orbit is the orbit of one the basic matrices: the
identity matrix in the first two cases, and $J$ for the third.
\par
We let $A = (a_{i j})$ denote an $m \times m$ complex matrix which may
be symmetric, general, or skew-symmetric.  As above, $A^{(k)}$ denotes 
the $k \times k$ upper left-hand corner submatrix. 
Then, these towers have the following properties.
\begin{Thm} \hfill
\label{CholTow}
\begin{itemize}
\item[i)] The tower of representations of $\{ B_m\}$ on $\{ \Sym_m\}$ is 
a
tower of block representations and the exceptional orbit varieties are
free divisors defined by $\prod_{k =1}^{m} \det(A^{(k)}) = 0$. 
\item[ii)]  The tower of representations of $\{ B_m \times N_m\}$ on $\{
M_m\}$ is a tower of non-reduced block representations and the
exceptional orbit varieties are free* divisors defined by $\prod_{k =1}^{m}
\det(A^{(k)}) = 0$.
\item[iii)] The tower of representations of $\{ D_m \}$ on $\{ \Sk_m\}$ is
a tower of non-reduced block representations and the exceptional orbit
varieties are free* divisors defined by $\prod_{k =1}^{\ell} \det(A^{(2k)})
= 0$, where $m = 2\ell$ or $2\ell+1$.
\end{itemize}
\end{Thm}
\par
\begin{Remark}  We make three remarks regarding this result. \par
 1)  Independently, Mond and coworkers \cite{BM}, \cite{GMNS}
in their work with reductive groups separately discovered the result for 
symmetric matrices by just directly applying the Saito criterion.  \par
2)  In the cases of general or skew-symmetric matrices, the exceptional 
orbit varieties are only free* divisors.  We will see in Theorems 
\ref{ModCholRep} and \ref{Thm.nonlinLiealg} that we can modify the 
solvable groups so the resulting representation gives a modified 
Cholesky-type factorization with exceptional orbit variety which still 
contains the variety of singular
matrices and which is a free divisor.  \par
3)  As a corollary of Theorem \ref{CholTow}, we deduce Cholesky-type 
factorization in the complex cases as exactly characterizing the elements 
belonging to the open orbit in each case.  The only point which has to be 
separately checked is the non-uniqueness, which is equivalent to 
determining the isotropy subgroup for the basic matrix in each case.
\end{Remark}
\par
\begin{proof}[Proof of Theorem \ref{CholTow}]
The proof will be an application of Proposition \ref{Prop7.2} for each of
the cases.  We begin with the case for the linear transformation
representation of $G_m = B_m \times N_m$ on $M_{m, m}$, the $m \times
m$ matrices.  We claim that the partial flag 
\begin{equation}
\label{Eqn.9.3a}
M \,\, = \,\,  M_{m, m}  \supset M^{(m-1, m-1)}   \supset  \dots M^{(1, 1)}
\supset M^{(0,0)}=0
\end{equation}
(using the notation of \S  \ref{S:MatrComps})
gives a nonreduced block representation.
By Lemma \ref{Lem8.4}, $K_{\ell} = K^{(\ell, \ell)}$
is the kernel of the quotient representation $\rho_{\ell, \ell} : G_m \to
\GL(M/M^{(\ell, \ell)})$.   We claim that together these give a nonreduced
block representation for $(G_m, M_{m, m})$.  \par
To show this, it is sufficient to compute the relative coefficient
matrix for the representation of $K_{\ell}/K_{\ell-1}$ on $M^{(\ell,
\ell)}/M^{(\ell - 1, \ell - 1)}$.  In fact, it is
useful to introduce
for $1\leq \ell< m$
a refinement of the decomposition by introducing
subrepresentations  $M^{(\ell, \ell)} \supset  M^{(\ell, \ell - 1)}  \supset
M^{(\ell - 1, \ell - 1)}$ in the sequence (\ref{Eqn.9.3a}), and the
corresponding kernels given by Lemma \ref{Lem8.4}
$K_{\ell} \supset K^{(\ell, \ell - 1)}  \supset K_{\ell-1}$.
\par
First, we consider the representation of $K_{\ell}/K^{(\ell, \ell-1)}$ on
$M^{(\ell, \ell)}/M^{(\ell, \ell - 1)}$.  To simplify notation, we let
$\ell^{\prime} = m - \ell$.  We use the complementary bases
$$  \{ E_{1\, \ell^{\prime} + 1},   E_{2\, \ell^{\prime} + 1}, \dots ,
E_{\ell^{\prime}\, \ell^{\prime} + 1}\} \quad \makebox{to  $M^{(\ell, \ell -
1)}$ in $M^{(\ell, \ell)}$, and } $$
$$  \{ (0,E_{1 \, \ell^{\prime} + 1}), (0,E_{2 \, \ell^{\prime} + 1}), \dots ,
(0,E_{\ell^{\prime}\, \ell^{\prime} + 1})\} \quad \makebox{ to $\bk^{(\ell,
\ell-1)}$ in $\bk^{(\ell, \ell)}$}. $$
Here $\bk^{(\ell, \ell)}/\bk^{(\ell, \ell-1)}$ is the Lie algebra of the
quotient group $N_m^{(\ell)}/N_m^{(\ell-1)}$, where $N_m^{(k)}$ denotes
the subgroup of $N_m$ consisting of matrices whose upper left $(m-k)
\times (m-k)$ submatrix is the identity.   \par
Using the notation of (\ref{Eqn8.8}) and \S  \ref{S:CholFact}, the
associated representation vector fields are  $\zeta_{j, \ell^{\prime} + 1} =
\xi_{(0, E_{j, \ell^{\prime} + 1})}$, $j = 1, \dots , \ell$.
Then, by using (\ref{Eqn8.9}), we compute the the relative coefficient
matrix with respect to the given bases and $A = (a_{i\, j})$
\begin{equation}
\label{Eqn9.4b}
 \zeta_{j, \ell^{\prime} + 1}(A)  \,\, = \,\, - \sum_{i = 1}^{m} a_{i, j}E_{i\,
\ell^{\prime} + 1} \, .
\end{equation}
Using (\ref{Eqn9.4b}), we see that with respect to the relative basis for
$M^{(\ell, \ell - 1)}$ in $M^{(\ell, \ell)}$ we obtain the relative coefficient
matrix $- (a_{i, j})$ for $i, j = 1, \dots \ell^{\prime}$.  For the $m \times
m$ matrix $A = (a_{i, j})$, this is the matrix $-A^{(\ell^{\prime})}$.
\par
Second, we consider the representation of $K^{(\ell, \ell-1)}/K_{\ell - 1}$
on $M^{(\ell, \ell - 1)}/M^{(\ell - 1, \ell - 1)}$.  Now we use the relative
bases
$$  \{ E_{\ell^{\prime} + 1\, 1},   E_{\ell^{\prime} + 1\, 2}, \dots ,
E_{\ell^{\prime} + 1\, \ell^{\prime} + 1}\} \quad \makebox{to  $M^{(\ell -
1, \ell - 1)}$ in $M^{(\ell, \ell - 1)}$, and } $$
$$  \{ (E_{\ell^{\prime} + 1\, 1},0), (E_{\ell^{\prime} + 1\, 2},0), \dots ,
(E_{\ell^{\prime} + 1\, \ell^{\prime} + 1},0)\} \quad \makebox{ to 
$\bk^{(\ell -
1, \ell-1)}$ in $\bk^{(\ell, \ell - 1)}$} \, . $$
Now $\bk^{(\ell - 1, \ell-1)}/\bk^{(\ell, \ell - 1)}$ is the Lie algebra of
the quotient group $B_m^{(\ell)}/B_m^{(\ell-1)}$, where $B_m^{(k)}$
denotes the subgroup of $B_m$ consisting of matrices whose upper left
$(m-k) \times (m-k)$ submatrix is the identity.  By (\ref{Eqn8.8}) the
associated representation vector fields  are $\xi_{\ell^{\prime} + 1, j} =
\xi_{(E_{\ell^{\prime} + 1, j}, 0)}$, $ j = 1, \dots , \ell^{\prime} + 1$.  An
argument analogous to the above using (\ref{Eqn8.9}) gives the relative
coefficient matrix to be the transpose of $A^{(\ell^{\prime} + 1)}$.
\par
Hence, we see that there will be contributions to the coefficient
determinant (up to a sign) of $\det A^{(\ell^{\prime})}$ twice appearing
for both $M^{(\ell, \ell - 1)} \subset M^{(\ell, \ell)}$ and $M^{(\ell, \ell)}$
in $M^{(\ell+ 1, \ell)}$.  Hence, the coefficient determinant is
$$  \prod_{k =1}^{m-1} \det (A^{(k)})^2  \cdot  \det (A)\, ,  $$
which is nonreduced.
\par
Next, for (i), we let $\Sym_m ^{(j, j)} = \Sym_m \cap M^{(j, j)}$.  By
Lemma \ref{Lem8.12}, these are invariant subspaces.  We claim that the 
partial flag
\begin{equation}
\label{eqn:syminvs}
\Sym_m \supset \Sym_m ^{(m-1,m-1)} \supset\cdots \cdots
\Sym_m ^{(1,1)}\supset 0
\end{equation}
gives a block representation of $B_m$ on $\Sym_m$.  By
Lemma \ref{Lem8.12} 
$$
K_\ell=\left\{\begin{pmatrix} \pm I_{m-\ell} & 0 \\ * & *
\end{pmatrix}
\in B_m\right\}
$$
is in the kernel of the quotient representation
$\rho_{\ell,\ell}:L_m\to\GL(\Sym_m/\Sym_m^{(\ell,\ell)})$; 
and an argument similar to that in the proof of Lemma \ref{Lem8.4} shows 
it is the entire kernel.  
Let $\ell^{\prime}=m-\ell$.  We let $e_{i\, j} = E_{i\, j} + E_{j\, i} \in 
\Sym_m$ and use the complementary bases
$$  \{ e_{1\, \ell^{\prime} + 1},  \dots , e_{\ell^{\prime}+1\, \ell^{\prime} 
+ 1}\} \quad 
\makebox{to $\Sym_m(\C)^{(\ell-1, \ell - 1)}$ in $\Sym_m(\C)^{(\ell, 
\ell)}$, and } $$
$$  \{ E_{\ell^{\prime} + 1\,1}, \dots , E_{\ell^{\prime}+1\, \ell^{\prime}+1 
}\} \quad \makebox{ to $\bk_{\ell-1}$ in $\bk_{\ell}$}\, . $$
By an analogue of \eqref{Eqn8.9}, but applied to \eqref{Eqn8.13}, 
the relative coefficient matrix with respect to these bases at
$A\in\Sym_m(\C)$ is $A^{(\ell^{\prime})}$.
Hence, the coefficient determinant is
\begin{equation}
\label{eqn:prodsym}
\prod_{\ell=1}^m \det(A^{(\ell)})\, .
\end{equation}
It only remains to show that \eqref{eqn:prodsym} is reduced.
We first show by induction on $\ell$
that each $p_\ell(A)=\det(A^{(\ell)})$ is irreducible.
Since $p_1$ is homogeneous of degree $1$, it is irreducible.
Assume by the induction hypothesis that $p_{\ell-1}$ is irreducible.
Expanding the determinant $p_\ell$ along the last column shows that its
derivative in the $E_{\ell,\ell}$ direction is $p_{\ell-1}$.
Since $p_{\ell-1}$ vanishes at \eqref{eqn:symmexample} and $p_{\ell}$
does not, $p_\ell$ is irreducible by Lemma \ref{Lem7.4}(ii).
\begin{equation}
\label{eqn:symmexample}
\begin{pmatrix}
I_{\ell-1} & & & \\
   & 0 & 1 & \\
   & 1 & 0 & \\
   & & & 0 \end{pmatrix}\in\Sym_m(\C) \, .
\end{equation}
Thus, \eqref{eqn:prodsym} is a factorization into irreducible
polynomials, and as each term $p_\ell$ has degree $\ell$, all terms
are relatively prime and hence \eqref{eqn:prodsym} is reduced.
\par
Lastly, consider (iii).  Though $D_m$ has a non-reduced block
representation using invariant subspaces having even-sized zero
blocks, it is easier to use a different group which has a finer
non-reduced block representation and the same open orbit.  Let $G_m$
be defined in the same way as $D_m$ but with $2\times 2$ diagonal
blocks of the form $\begin{pmatrix} 1 & 0 \\ 0 & r \end{pmatrix}$,
$r\neq 0$.  We claim that the partial flag
\begin{equation}
\label{eqn:skewblockspaces}
\Sk_m(\C) \supset \Sk_m(\C)^{(m-1,m-1)}\supset \cdots \supset
\Sk_m(\C)^{(1,1)}\supset 0
\end{equation}
gives a non-reduced block representation of $G_m$.  By Lemma
\ref{Lem8.12}, \eqref{eqn:skewblockspaces} are invariant subspaces
and 
\begin{equation}
K_\ell=\left\{\begin{pmatrix} \pm I_{m-\ell} & 0 \\ * & *
\end{pmatrix}
\in G_m\right\}
\end{equation}
is in the kernel of the quotient representation
$\rho_{\ell,\ell}:G_m\to\GL(\Sk_m(\C)/\Sk_m(\C)^{(\ell,\ell)})$.

\begin{figure}
$$
\begin{matrix} 
   &
     \begin{array}{ccc}   
       \overbrace{
         \hphantom{\begin{matrix} * \;\;\cdots\;\;* \end{matrix}}
       }^{ \ell^{\prime}}
  &
       \overbrace{
         \hphantom{\begin{matrix} 1\quad 0\end{matrix}}
       }^{2}
  &
        \hphantom{\begin{matrix} 1 & 0 \end{matrix}}
     \end{array}
   \\
     \begin{array}{r}    
       \textrm{$\ell^{\prime}$ even} \left\{\vphantom{\begin{matrix} 0 \\ 0 
\end{matrix}}\right. \\
       2 \left\{\vphantom{\begin{array}{c} * \\ * \end{array}}\right. \\
       \vphantom{ \begin{matrix} 0 \\ 0 \end{matrix}}
     \end{array}
     \mspace{-25mu}
   &
     \left(\begin{array}{c|c|c}
       \phantom{\begin{matrix} 0 \\ 0 \end{matrix}} & & \\
 	\hline
       * \;\;\cdots\;\;* & 1\quad 0 & \\
       * \;\;\cdots\;\;* & 0\quad * & \\
 	\hline
         & & \phantom{\begin{matrix} 1 & 0 \\ 0 & 1 \end{matrix}}\\
     \end{array}\right)
\end{matrix}
$$
\caption{The group $G_m$ used in the proof of Theorem
\ref{CholTow}(iii).}
\label{fig:ellprime}
\end{figure}
 We let $\overline{e}_{i\, j} = E_{i\, j} - E_{j\, i} \in \Sk_m(\C)$ for $ 1 
\leq i < j 
\leq m$ and let $\ell^{\prime}=m-\ell$.  We see in Figure 
\ref{fig:ellprime} the form of $G_m$, and obtain the resulting 
complementary bases.  \par
When $\ell^{\prime}$ is even, we use the
complementary bases
$$  \{ \overline{e}_{1\, \ell^{\prime} + 1},   \dots ,
\overline{ e}_{\ell^{\prime}\, \ell^{\prime} +
1}\} \quad \makebox{to
$\Sk_m(\C)^{(\ell-1, \ell - 1)}$ in $\Sk_m(\C)^{(\ell, \ell)}$, and } $$
$$  \{ E_{\ell^{\prime} + 1\,1}, E_{\ell^{\prime} + 1\, 2}, \dots ,
E_{\ell^{\prime}+1\, \ell^{\prime} }\} \quad \makebox{ to
$\bk^{(\ell-1,
\ell-1)}$ in $\bk^{(\ell, \ell)}$}. $$
By an analogue of \eqref{Eqn8.9} for \eqref{Eqn8.13}, we find that at
$A=(a_{i\,j})\in\Sk_m(\C)$, the relative coefficient matrix for these
bases is $A^{(\ell^{\prime})}$.
Its determinant is the square of the Pfaffian $\Pf(A^{(\ell^{\prime})})$. 
\par
When $\ell^{\prime}$ is odd, we use the
complementary bases.
$$  \{ \overline{ e}_{1\, \ell^{\prime} + 1},   \dots ,
\overline{e}_{\ell^{\prime}\, \ell^{\prime} +
1}\} \quad \makebox{to
$\Sk_m(\C)^{(\ell-1, \ell - 1)}$ in $\Sk_m(\C)^{(\ell, \ell)}$, and } $$
$$  \{ E_{\ell^{\prime} + 1\,1}, \dots ,
E_{\ell^{\prime}+1\, \ell^{\prime}-1},
E_{\ell^{\prime}+1\, \ell^{\prime}+1}\} \quad \makebox{ to
$\bk^{(\ell-1,
\ell-1)}$ in $\bk^{(\ell, \ell)}$}. $$
\par

We find that the resulting relative coefficient matrix for these bases
is $A^{(\ell^{\prime}+1)}$ with column $\ell^{\prime}$ and row
$\ell^{\prime}+1$ deleted.
Its determinant factors as the product
of Pfaffians, $\Pf(A^{(\ell^{\prime}+1)})\Pf(A^{(\ell^{\prime}-1)})$
(see \cite{MM}, \S 406-415).
Hence, the coefficient determinant is nonreduced,
with factorization
$$
\left(\prod_{i=1}^{k-1} \Pf(A^{(2i)})^4\right)
\Pf(A^{(2k)})
\qquad\textrm{or}\qquad
\left(\prod_{i=1}^{k-1} \Pf(A^{(2i)})^4\right)
\Pf(A^{(2k)})^3
$$
when $m=2k$ or $m=2k+1$, respectively.  
\par
We now show that $G_m$ and $D_m$ have the same open orbit. 
Let $J$ be the matrix from Theorem \ref{ComplCholFac} (3), an element of 
the open orbit of $G_m$.  Let $K$ be the group of
invertible $m\times m$ diagonal matrices with $2\times 2$ diagonal
blocks in $\SL_2(\C)$ (with a last entry of $1$ if $m$ is odd).
Easy calculations show that $K$ lies in the
isotropy group at $J$, and that for all $A\in G_m$ (resp., all $B\in
D_m$), there exists a $C\in K$ so that $AC\in D_m$ (resp., $BC\in
G_m$); thus $AJA^T=ACJ(AC)^T$ (resp., $BJB^T=BCJ(BC)^T$), and
$G_m$ and $D_m$ have the same open orbit.
\end{proof}

\section{Modified Cholesky-Type Factorizations as Block
Representations}
\label{S:ModCholreps}  
\par
In the previous section we saw that for both general $m \times m$
matrices and skew--symmetric matrices, the corresponding exceptional
orbit varieties are only free* divisors.  In this section we address the
first case by considering a modification of the Cholesky-type
representation for general $m \times m$ matrices.  This further extends
to the space of $(m-1) \times m$ general matrices.  In each case there 
will result a modified form of Cholesky-type factorization.  \par
\vspace{1ex}
\flushpar
{\it General $m \times m$ complex matrices : }
For general $m \times m$ complex matrices we let $C_m$ denote the
subgroup of invertible upper triangular matrices with first diagonal entry
$ = 1$ and other entries in the first row $0$.  $C_m$ is naturally
isomorphic to $B_{m-1}^T$ via
\begin{align}
B_{m-1}^T \,\,  &\longrightarrow \,\, C_m  \\
B  \quad &\mapsto  \begin{pmatrix}
1 & 0 \\ 0  &  B
\end{pmatrix} \,\, . \notag
\end{align}
We consider the action of $B_m \times C_m$ on $V = M_{m, m}$ by $(B,
C)\cdot A = B\, A\, C^{-1}$.  This is the restriction of the linear
transformation representation.  We again have the natural inclusions
$M_{m, m} \hookrightarrow M_{m+1, m+1}$  and $B_m \times C_m
\hookrightarrow B_{m+1} \times C_{m+1}$ where the inclusions (of each
factor) are as in (\ref{Eqn9.2a}).
These inclusions define a tower of representations of $\{ B_m \times
C_m\}$ on $\{ M_{m, m}\}$.
\par
\vspace{1ex} \flushpar
{\it General $(m-1) \times m$ complex matrices : }
We modify the preceding action to obtain a representation of $B_{m-1}
\times C_m$ on $V = M_{m-1, m}$ by $(B, C)\cdot A = B\, A\, C^{-1}$.  We
again have the natural inclusions $M_{m-1, m} \hookrightarrow M_{m,
m+1}$ as in (\ref{Eqn9.2a}).  Together with the natural inclusions $B_{m-
1} \times C_m \hookrightarrow B_{m} \times C_{m+1}$, we again obtain a
tower of representations of $\{ B_{m-1} \times C_m\}$ on $\{ M_{m-1,
m}\}$.
\\

To describe the exceptional orbit varieties, for an $m \times m$ matrix
$A$, we let $\hat A$ denote the $m \times (m-1)$ matrix obtained by
deleting the first column of $A$.  If instead $A$ is an $(m-1) \times m$
matrix, we let $\hat A$ denote the $(m-1) \times (m-1)$ matrix obtained
by deleting the first column of $A$.  In either case, we let $\hat A^{(k)}$
denote the $k \times k$  upper left submatrix of $\hat A$, for $1 \leq k
\leq m-1$.  Then, the towers of modified Cholesky-type representations
given above have the following properties.

\begin{Thm}[Modified Cholesky-Type Representation] \hfill
\label{ModCholRep}
\begin{enumerate}
\item {\em Modified LU decomposition: } The tower of representations $\{
B_m \times C_m\}$ on $\{ M_{m, m}\}$ has a block representation and the
exceptional orbit varieties are free divisors defined by $$\prod_{k =1}^{m}
\det(A^{(k)})\cdot \prod_{k =1}^{m-1}\det(\hat A^{(k)}) \,\, = \,\, 0\, . $$
\item  {\em Modified Cholesky representation for $(m-1) \times m$
matrices: } The tower of representations $\{ B_{m-1} \times C_m\}$ on
$\{ M_{m-1, m}\}$ has a block representation and the exceptional orbit
varieties are free divisors defined by
$$\prod_{k =1}^{m-1} \det(A^{(k)})\cdot \prod_{k =1}^{m-1}\det(\hat
A^{(k)}) \,\, = \,\, 0 \, . $$
\end{enumerate}
\end{Thm}
\par
\begin{proof}
For i), we let $\tau$ denote the restriction of $\rho$ to $G_m = B_m 
\times C_m$.  We will apply Proposition \ref{Prop7.2} using the same 
chain of invariant subspaces $\{ W_j\}$ of $M = M_{m, m}$ formed from 
$M^{(\ell, \ell)}$ and the refinements obtained by introducing the 
intermediate subspaces $M^{(\ell, \ell-1)}$ used in the proof of ii) in 
Theorem \ref{CholTow}.  We let $K^{(\ell, \ell-1)}$ denote the 
corresponding kernels for $G_m = B_m \times C_m$ acting on $M_{m, m}$.  
Because the group $B_m$ is unchanged the computation for the 
representation of $K^{(\ell, \ell-1)}/K_{\ell - 1}$  on $M^{(\ell, \ell - 
1)}/M^{(\ell - 1, \ell -1)}$ is the same as in ii) of Theorem \ref{CholTow}.  
\par
We next have to replace the calculation for $N_m$ by that for $C_m$ for
the representation of $K_{\ell}/K^{(\ell, \ell-1)}$ on $M^{(\ell,
\ell)}/M^{(\ell, \ell - 1)}$.  We note that this changes exactly one vector in
the basis, replacing $E_{\ell^{\prime} + 1\, 1}$ by  $E_{\ell^{\prime} + 1\,
\ell^{\prime} + 1}$.  When we compute the associated representation
vector field, we obtain the column vector formed from the first
$\ell^{\prime}$ entries of the $\ell^{\prime} + 1$ column of $A$.  Hence,
we remove the first column and replace it by the  $\ell^{\prime} + 1$-st
column.  This is exactly the matrix  $-(\hat A)^{(\ell^{\prime})}$.
Hence the coefficient determinant is (up to a sign)
\begin{equation}
\label{eqn:modluprod}
\prod_{j = 1}^{m} \det (A^{(j)})  \cdot \prod_{k = 1}^{m-1} \det (\hat
A^{(k)})\, . 
\end{equation} 
\par
\vspace{1ex}
\par
We now show \eqref{eqn:modluprod} is reduced.  We proceed by induction
on the size of the determinant.  The functions
$A\mapsto \det(A^{(1)})$ and
$A\mapsto \det({\hat A}^{(1)})$ are irreducible
since they are homogeneous of degree $1$.
Suppose $A\mapsto \det(A^{(k)})$ (respectively, 
$A\mapsto \det({\hat A}^{(k)})$) are irreducible for $k < j$. These 
determinants are related by differentiation:  
$$\pd{\det(A^{(j)}) }{a_{j\, j}} \, = \, \det(A^{(j-1)})  \qquad \makebox{ and 
}  \qquad  \pd{ \det({\hat A}^{(j)}) }{a_{j\, j+1}} \, = \,  \det({\hat A}^{(j-
1)})  \, .   $$
Thus, we may apply
Lemma \ref{Lem7.4}(ii), using the induction hypothesis, to 
\eqref{eqn:modluirredexs}a)
(respectively, \eqref{eqn:modluirredexs}b) )
and deduce that the $j \times j$ determinants are irreducible.
\begin{equation}
\label{eqn:modluirredexs}
a)
\begin{pmatrix}
I_{j-2} & & & \\
 & 0 & 1 & \\
 & 1 & 0 & \\
 & & & 0 \end{pmatrix}
\qquad
b)
\begin{pmatrix}
0_{(j-2) \times 1} & I_{j-2} & & & \\
& & 0 & 1 & \\
& & 1 & 0 & \\
& & & & 0 \end{pmatrix}
\end{equation}
Thus, each factor of \eqref{eqn:modluprod} is irreducible.
Based on the (polynomial) degrees of $A\mapsto \det(A^{(j)})$ and
$A\mapsto \det({\hat A}^{(j)})$
 and their values at \eqref{eqn:modluirredexs}a), we conclude the factors 
are irreducible and distinct; hence, \eqref{eqn:modluprod} is reduced.
\par
Hence, the modified Cholesky-type representation on $m \times m$
complex matrices is a block representation.  Furthermore, the induced
quotient representation of $G_m = B_m \times C_m$  on $M_{m, m}/M^{(1,
1)}$ has kernel $K_1$ and it is easy to check that $G_m/K_1 \simeq
G_{m-1}$.  Hence, the $(G_m, M_{m, m})$ form a tower of block
representations.
\par
We obtain ii) of the theorem by modifying the proof of i).  Now $G_m = 
B_{m-1} \times C_m$ is acting on $M = M_{m-1, m}$, and 
when using the intermediate subspaces $M^{(\ell,\ell-1)}$, the
last nontrivial group and subspace in the block structure for
$(G_m,M_{m,m})$ is $K^{(m-1,m)} \subset G_m = B_{m-1} \times C_m$ and 
$M_{m,m}^{(1,0)}$, whose relative coefficient determinant is the 
determinant function.  \par
By Proposition \ref{prop5.1}, the representation of $G_m/ K^{(1, 0)}$ on 
the quotient $M_{m,m}/M_{m,m}^{(1,0)}$ gives a block representation 
isomorphic to the one described.  In turn, the block representation for 
$M_{m-1,m}$ has $M_{m-1,m}^{(1,1)}$ as an invariant subspace with 
$K^{(1,1)}$ the kernel of the induced quotient representation.  Forming the 
quotient $M_{m-1,m}/M_{m-1,m}^{(1,1)}$ gives a block representation of 
$G_m/K^{(1,1)}$ isomorphic to the one on $M_{m-2,m-1}$.
Hence, we obtain a tower of block representations.
\end{proof}
We have the following consequences for modified forms of Cholesky-type
factorizations which follow from Theorem \ref{ModCholRep}.

\begin{Thm}[Modified Cholesky-Type Factorization] \hfill
\label{ModCholFac}
\begin{enumerate}
\item {\em Modified LU decomposition: } If $A$ is a general complex $m
\times m$ matrix with $\det (A^{(k)}) \neq 0$  for $k = 1, \dots, m$ and
$\det (\hat A^{(k)}) \neq 0$  for $k = 1, \dots, m-1$, then there exists a
unique lower triangular matrix $B$ and a unique  upper triangular matrix
$C$, which has first diagonal entry $= 1$, and remaining first row entries
$= 0$ so that $A = B \cdot K \cdot C$, where $K$ has the form of a) in
(\ref{Eqn9.4}).
\item  {\em Modified Cholesky factorization for $(m-1) \times m$
matrices: } If $A$ is an $(m-1) \times m$ complex matrix with $\det
(A^{(k)}) \neq 0$  for $k = 1, \dots, m - 1$, $\det (\hat A^{(k)}) \neq 0$  for
$k = 1, \dots, m-1$, then there exists a unique $(m-1) \times (m-1)$
lower triangular matrix $B$ and a unique $m \times m$ matrix $C$ having
the same form as in (1), so that $A = B \cdot K^{\prime} \cdot C$, where
$K^{\prime}$ has the form of b) in (\ref{Eqn9.4}).
\end{enumerate}
\begin{equation}
\label{Eqn9.4}
a) \qquad \begin{pmatrix}
1 &1 & 0 & 0 & 0 \\ 0  & 1 & 1 & 0 & 0 \\  0  & 0 & \ddots & 1 & 0 \\
0  & 0 & 0 & 1 & 1 \\  0  & 0  & 0  & 0 & 1
\end{pmatrix}  \qquad \makebox{ and } \quad b) \qquad    \begin{pmatrix}
1 &1 & 0 & 0 & 0 & 0 \\ 0  & 1 & 1 & 0 & 0 & 0 \\  0  & 0 & \ddots & 1 & 0
& 0 \\
0  & 0 & 0 & 1 & 1 & 0 \\  0  & 0  & 0 & 0  & 1 & 1
\end{pmatrix}
\end{equation}
\end{Thm}
\par
The factorization theorem follows from Theorem \ref{ModCholRep} by
directly checking that the matrices a), respectively b), in (\ref{Eqn9.4}) 
are
not in the exceptional orbit varieties.
\par
We summarize in Table \ref{table2.0}, each type of complex (modified)
Cholesky-type representation, the space of complex matrices, the
solvable group and the representation type. 

\begin{table}
\begin{tabular}{lccc}
Cholesky-type  & Matrix Space & Solvable Group & Representation \\
\quad Factorization  &     &      &     \\
\hline
Symmetric matrices & $\Sym_m$ & $B_m$  &  Bil\\
General  matrices  &$M_{m, m}$ & $B_m \times N_m$  &  LT\\
Skew-symmetric &  $\Sk_m$  & $D_m$  &  Bil  \\
\hline \hline
Modified Cholesky  &     &    &     \\
\quad -type Factorization  &     &      &     \\
\hline
General $m \times m$ & $M_{m, m}$  &  $B_m \times C_m$   &   LT   \\
General $(m-1) \times m$  & $M_{m-1, m}$ & $B_{m-1} \times C_m$   &
LT     \\
\end{tabular}
\vspace*{.2cm}
\caption{\label{table2.0} Solvable groups and (nonreduced) Block
representations for (modified) Cholesky-type Factorization arising from
either the linear transformation representation (LT) or bilinear
representation (Bil).}
\end{table}

\section{Block Representations for Nonlinear Solvable Lie Algebras}
\label{S:NonlinLie}
In the preceding section we saw that the Cholesky-type representations 
for the spaces of general $m \times m$ and $m \times (m+1)$ matrices 
were nonreduced block representations, yielding free* divisors.  However, 
by modifying the solvable groups and representations we obtained block 
representations, whose exceptional orbit varieties are free divisors and 
contain the determinantal varieties.  In this section, we take a different 
approach to modifying the Cholesky representation on $\Sk_{m}(\C)$ to 
obtain a representation whose exceptional orbit variety is a free divisor 
containing the Pfaffian variety.  The underlying reason for this change is 
that factorization properties of determinants of submatrices of skew-
symmetric matrices suggests that a reduced exceptional orbit variety may 
not be possible for a solvable linear algebraic group.  However, the 
essential ideas of the block representation will continue to be valid if we 
replace the finite dimensional solvable Lie algebra by an infinite 
dimensional solvable holomorphic Lie algebra which has the analog of a 
block representation.  \par 
 We will then obtain the exceptional orbit varieties which are 
\lq\lq nonlinear\rq\rq free divisors.  The resulting sequence of free 
divisors on $\Sk_m(\C)$ (for all $m$) have the tower-like property that 
they are formed by repeated additions of generalized determinantal and 
Pfaffian varieties (c.f., Theorem \ref{Towerthm}(iii)).
We shall present the main ideas here, but we will refer to \S 5.2 of 
\cite{P} for certain technical details of the computations.   \par

We first consider the bilinear form representation on $\Sk_m(\C)$ of
the group 
\begin{equation}
\quad  G_m =  \left\{
\begin{pmatrix}
T_{2} & 0_{2, m-2} \\ 0_{m-2, 2}  &  B_{m-2}
\end{pmatrix} \right\},
\end{equation}
where $T_2$ is the group of $2\times 2$ invertible diagonal matrices.
Let $\g_m$ be the Lie algebra of $G_m$.
When $m=3$, the exceptional orbit variety of this representation is
the normal crossings linear free divisor on $\Sk_3(\C)$.
For $m>3$, this representation cannot have an open orbit, as
$\dim(\Sk_m(\C))-\dim(G_m)=m-3$.  Nonetheless, this is a 
representation of the finite dimensional solvable Lie algebra $\g_m$ on 
$\Sk_m(\C)$.  The associated representation vector fields generate a 
solvable holomorphic Lie algebra $\cL_m^{(0)}$.  Our goal is to construct 
an extension of $\cL_m^{(0)}$ by adjoining as generators $m-3$ nonlinear 
Pfaffian vector fields to obtain a solvable holomorphic Lie algebra 
$\cL_m$ which is a 
free $\cO_{s_m}$ module of rank $s_m = \dim_{\C} \Sk_m(\C)  = 
\tbinom{m}{2}$, where we abbreviate $\cO_{\Sk_m(\C), 0}$ as 
$\cO_{s_m}$.  Then we will apply Saito\rq s criterion to deduce that the 
resulting \lq\lq exceptional orbit variety\rq\rq is a free divisor. 
\par
For $S\subseteq \{1,\ldots,m\}$ and $A\in\Sk_m(\C)$, we define 
$\Pf_S(A)$ to be the Pfaffian of the matrix obtained by deleting all rows 
and columns of $A$ not indexed by $S$.
For any $i\in\{2,\ldots,m\}$, let $\epsilon(i)$ be either $1$ or
$2$, so that $\epsilon(i)$ and $i$ have opposite parity, and hence  
$\{\epsilon(i),\epsilon(i)+1,\ldots,i\}$ has even cardinality.
As in \S \ref{S:CholFact}, we let $\overline{e}_{i,j} = E_{i,j} -E_{j,i} \in 
\Sk_m(\C)$ for $1 \leq i < j \leq m$.  Then for $2\leq k\leq m-2$, define
\begin{equation}
\eta_k(A)=\sum_{k<p<q\leq m} \Pf_{\{\epsilon(k),\ldots,k,p,q\}}(A)
 \cdot \overline{e}_{p,q},
\end{equation}
which is a (homogeneous) vector field on $\Sk_m(\C)$ of degree $\lfloor 
\frac{k}{2}\rfloor$.  Here $\overline{e}_{p,q}$, viewed as a constant 
vector 
field, denotes $\pd{ }{a_{p, q}} - \pd{ }{a_{q, p}}$ and hence has degree $-
1$. \par
For example, if $m = 2\ell$, the degrees of the $\eta_k$ form a sequence 
$1, 1, 2, 2, \dots$, ending with a single top degree $\ell - 1$; while for $m 
= 2\ell +1$, the sequence consists of successive pairs of integers. For $m$ 
even, the top vector field is just $\Pf(A) \,\overline{e}_{m-1, m}$.  \par 
Then, $\cL_m$ will be the $\cO_{s_m}$-module generated by a basis $\{ 
\xi_{E_{i, j}}\}$ of representation vector fields associated to $G_m$ and 
$\{ \eta_k$, $2\leq k\leq m-2\}$.  Note this module has $s_m$ generators 
so Saito\rq s  criterion may be applied.  We let, as earlier, $\hat A$ 
denote the matrix $A$ with the left column removed, and let $\hat{\hat 
A}$ be the matrix $A$ with the two left columns deleted.  \par
Then, the modification of the Cholesky-type representation for the 
$\Sk_m(\C)$ is given by the following result.
\begin{Thm}
\label{Thm.nonlinLiealg}
The $\cO_{s_m}$ module $\cL_m$ is a solvable holomorphic Lie algebra 
for $m \geq 3$.  In addition, it is a free $\cO_{s_m}$ module of rank 
$s_m$, and it defines a free divisor on $\Sk_m(\C)$ given by the equation 
\begin{equation}
\label{eqn:skewnonlinFD0}
\prod_{k=1}^{m-2} \det\left({\hat{\hat A}}^{(k)}\right) \cdot
\prod_{k=2}^m \Pf_{\{\epsilon(k),\ldots,k\}}(A) \, = \, 0  \,\, .
\end{equation}
\end{Thm}
\begin{Remark}
\label{Rem.11.1}
We note in (\ref{eqn:skewnonlinFD0}), that when $k$ is odd, $\epsilon(k) = 
2$, so that $\Pf_{\{\epsilon(k),\ldots,k\}}(A)$ is the Pfaffian of the 
$(k-1) \times (k-1)$ upper left-hand submatrix of the matrix obtained 
from $A$ by first deleting the top row and first column.
\end{Remark}
\par
Before proving this theorem, we illustrate it in the simplest nontrivial 
case of $\Sk_4(\C)$.
\begin{Example}
\label{Ex.nonlinalg}
First,  $\dim_{\C} \Sk_4(\C) = 6$; while $\dim G_4 = 5$, with Lie algebra 
$\bg_4$ having basis $\{ E_{1, 1}, E_{2, 2}, E_{3, 3}, E_{4,3}, E_{4,4}\}$. 
For $\cL_4$, we adjoin to the representation vector fields associated to 
the basis for $\bg_4$ an additional generator $\eta_2 = \Pf(A)\cdot
\overline{e}_{3, 4}
\, (= \Pf(A) \cdot (\pd{ }{a_{3, 4}} -\pd{ }{a_{4, 3}}))$.  Then the 
coefficient matrix using the basis $\{\overline{e}_{1, 2},
\overline{e}_{1, 3},
\overline{e}_{2, 3}, \overline{e}_{1, 4}, \overline{e}_{2, 4}, 
\overline{e}_{3, 4}\}$ is \par  
$$
\begin{pmatrix}
a_{12} & a_{12} & 0 & 0 & 0 & 0\\
a_{13} & 0 & a_{13} & 0 & 0 & 0\\
0 & a_{23} & a_{23} & 0 & 0 & 0\\
a_{14} & 0 & 0 & a_{13} & a_{14} & 0 \\
0 & a_{24} & 0 & a_{23} & a_{24} & 0\\
0 & 0 & a_{34} & 0 & a_{34} &\Pf (A)
\end{pmatrix}
$$

\flushpar
which has block lower triangular form, with determinant 
$$a_{1\, 2} a_{1\, 3} a_{2\, 3} (a_{1\, 3} a_{2\, 4} - a_{1\, 4}a_{2\, 
3})\cdot \Pf(A) \, .$$  
The term $a_{2\, 3}$ is the Pfaffian $\Pf_{\{2, 3\}}(A)$ as described in 
Remark \ref{Rem.11.1}.  The determinant has degree $7$ and, by the 
theorem, defines a free divisor, which is not a linear free divisor.  
\end{Example}
\begin{proof}[Proof of Theorem \ref{Thm.nonlinLiealg}] 
To prove the theorem we will apply Saito\rq s Criterion (Theorem 
\ref{Saitocrit}(2)).  For it, we first 
show that $\cL_m$ is a holomorphic Lie algebra.
Since $\g_m$ is a Lie algebra, it is sufficient to show that both 
$[\xi,\eta_k]$ and $[\eta_k,\eta_l]\in\cL_m$ for all $2 \leq  \ell, k \leq 
m-2$ and any representation vector field $\xi$ associated to $G_m$.
\begin{Proposition}
\label{prop:nonlinvfbracket}
If $E_{p,q}\in\g_m$, then
$$[\xi_{E_{p,q}},\eta_k]=\left\{
\begin{array}{ll}
\eta_k & \textrm{if $p=q$ and $\epsilon(k)\leq p\leq k$}
\\
0 & \textrm{otherwise}
\end{array}\right..$$
If $k<l$, then
$$[\eta_k,\eta_l]=
\frac{1}{2}(\delta_{\epsilon(k),\epsilon(l)}+l-k-1)
\Pf_{\{\epsilon(k),\ldots,k\}}\cdot \eta_l.$$
\end{Proposition}
\begin{proof}
The full details are given in Appendix A of \cite{P}.
However, we remark that the computation of these Lie brackets is very 
lengthy, and makes repeated applications of the following Pfaffian 
identity of Dress-Wenzel.  
\end{proof}
\begin{Thm}[Dress-Wenzel \cite{DW}]
\label{thm:dw}
Let $I_1,I_2\subseteq \{1, \ldots, m\}$.  Write the
symmetric difference $I_1\Delta I_2=\{i_1,\ldots,i_\ell\}$ with 
$i_1 < \cdots < i_\ell$.  Then
$$
\sum_{\tau=1}^\ell (-1)^\tau \Pf_{I_1\Delta\{i_\tau\}}
\Pf_{I_2\Delta\{i_\tau\}}=0.$$
\end{Thm}
\par  
We next show that $\cL_m$ is free as an $\cO_{s_m}$-module.  To do this, 
we determine the coefficient matrix of the generators of $\cL_m$. \par
By the discussion in \S \ref{S:MatrComps} and \S \ref{ComplCholFac}, the 
bilinear form representation has the invariant subspaces
$\Sk_m(\C)^{(\ell,\ell)} = \Sk_m(\C) \cap M^{(\ell, \ell)}$, and the
kernels of the induced quotient representations for $0\leq \ell\leq m-3$ 
are
\begin{equation}
\label{Eqn9.9}
K_\ell=\left\{\begin{pmatrix}
\pm I_{m-\ell} & 0 \\ * & B_\ell \end{pmatrix}\in G_m\right\}.
\end{equation}
(The kernels for $\ell=m-2,m-1$ do not take this form.) We denote the Lie 
algebras of $K_{\ell}$ by $\bk_{\ell}$.  \par
For the decomposition, we consider $\Sk_m(\C)^{(\ell,\ell)}$ for $0\leq 
\ell\leq m-3$ (together with $\Sk_m(\C)$).  First, the complementary 
basis for $\Sk_m(\C)^{(m-3, m-3)}$ in $\Sk_m(\C)$ is $\{ \overline{e}_{1, 
2}, 
\overline{e}_{1, 3}, \overline{e}_{2, 3}\}$, and $\{ E_{1, 1}, E_{2, 2}, E_{3, 
3}\}$ is a 
complementary basis for $\bk_{m-3}$ in $\bg_m$.
\par 
For $\ell \leq m-3$, as earlier we let $\ell^{\prime} = m-\ell$, and use 
the complementary bases
$$  \{ \overline{e}_{1\, \ell^{\prime} + 1},   \dots ,
\overline{e}_{\ell^{\prime}\, \ell^{\prime} +
1}\} \quad \makebox{to
$\Sk_m(\C)^{(\ell-1, \ell - 1)}$ in $\Sk_m(\C)^{(\ell, \ell)}$. } $$
For the subgroups $K_\ell$, we use the corresponding complementary 
bases
$$  \{ E_{\ell^{\prime} + 1\,3}, \dots ,
E_{\ell^{\prime}+1\, \ell^{\prime}+1}\} \quad \makebox{ to
$\bk_{\ell-1}$ in $\bk_{\ell}$}. $$
\par
As 
$$\dim (\Sk_m(\C)^{(\ell, \ell)}/\Sk_m(\C)^{(\ell-1, \ell - 1)}) \,\, =  \,\, 
\ell^{\prime} \,\, =  \,\, \dim \bk_{\ell}/\bk_{\ell-1} + 1 \, , $$
we adjoin a single $\eta_k$ with $k = m -\ell-1 = \ell^{\prime} -1$.  We 
note that just as for $\xi_{ E_{\ell^{\prime} + 1\,j}}$, this 
$\eta_{\ell^{\prime}-1}$ has $0$ coefficients for the relative basis of 
$\Sk_m(\C)/\Sk_m(\C)^{(\ell, \ell)}$.  
\begin{Proposition}
\label{prop:nonlincoeff}
With the above relative bases (with the corresponding 
$\eta_{\ell^{\prime} -1}$ adjoined to the appropriate relative bases as 
indicated) the coefficient matrix of 
$\cL_m$ is block lower triangular with $m-2$ diagonal blocks
$\{D_{\ell}\}$ (as in \eqref{matr4.5}), where
at $A=(a_{ij})\in\Sk_m(\C)$,
$$
D_{m-2}(A)=\begin{pmatrix}
a_{12} & a_{12} & 0 \\
a_{13} & 0 & a_{13} \\
0 & a_{23} & a_{23}
\end{pmatrix} $$
and for $1 \leq \ell \leq m-3$, with $\ell^{\prime} = m - \ell$, there is 
the $\ell^{\prime} \times \ell^{\prime}$ diagonal block
\begin{equation}
\label{eqn:skewdiagblock}
D_{\ell}(A) =
\begin{pmatrix} 
a_{1,3} & \cdots & a_{1,\ell^{\prime} + 1}  & 0\\
\vdots & \ddots & \vdots & \vdots \\
a_{\ell^{\prime}-1,3} & \cdots & a_{\ell^{\prime} - 1, \ell^{\prime} + 1} & 
0 \\
a_{\ell^{\prime}, 3} & \cdots & a_{\ell^{\prime}, \ell^{\prime} + 1} 
& \Pf_{\{\epsilon(\ell^{\prime} - 1),\ldots, \ell^{\prime} - 1\}}(A)
\end{pmatrix}.
\end{equation}
Hence, the coefficient determinant for this block is
\begin{equation}
\label{eqn:skewcoeffdet}
 \det (D_{\ell}(A)) \,\, = \,\, \det (\hat{\hat A}^{(\ell^{\prime}-1)}) \cdot 
\Pf_{\{\epsilon(\ell^{\prime} - 1),\ldots, \ell^{\prime} - 1\}}(A)\, . 
\end{equation}
\end{Proposition}
\begin{proof}
We claim that the coefficient matrix with respect to the two sets of 
bases is block lower triangular with $m-2$ blocks.  The first block 
corresponds to $\bg_m/K_{m-3}$ and a direct calculation shows it is the 
$3 \times 3$ block $D_{m-2}$ in the proposition.  
For the subsequent blocks, we note by Lemma \ref{Lem4.1} and the remark 
concerning $\eta_k$ preceding the proposition, that the columns 
corresponding to $\{ E_{\ell^{\prime} + 1\,3}, \dots , E_{\ell^{\prime}+1\, 
\ell^{\prime}+1}, \eta_{\ell^{\prime} - 1} \}$ will be $0$ above the 
$\ell^{\prime} \times \ell^{\prime}$ diagonal block $D_{\ell}$.  \par
Moreover, for this block, by the calculations carried out in \S 
\ref{S:CholFact},
the upper left $(\ell^{\prime}-1) \times (\ell^{\prime}-1)$ submatrix is  
$\hat{\hat A}^{(\ell^{\prime}-1)}$ (because $E_{\ell^{\prime} + 1\,1}$ and 
$E_{\ell^{\prime} + 1\,2}$ are missing in the basis for 
$\bk_{\ell}/\bk_{\ell-1}$).  Also, by the form of $\eta_{\ell^{\prime} - 
1}$, the column for it will only have an entry 
$\Pf_{\{\epsilon(\ell^{\prime} - 1),\ldots, \ell^{\prime} - 1\} }$ in the 
last row of the block.  Thus, $D_{\ell}$ and $\det (D_{\ell})$ have the 
forms as stated.
\end{proof}
\par
Then, applying Proposition \ref{prop:nonlincoeff} to each diagonal block 
yields as the coefficient determinant (up to sign) the left-hand side of 
(\ref{eqn:skewnonlinFD0}).  
Lemma \ref{Lem7.4} can be used as in earlier cases to show that the 
determinant is reduced.  Thus, by Saito's Criterion $\cL_m$ is a free 
$\cO_{s_m}$ module which defines a free divisor on $\Sk_m(\C)$ with 
defining equation (\ref{eqn:skewnonlinFD0}).  
\par
Lastly, since the degree $0$ subalgebra $\bg_m$ of $\cL_m$ is solvable, 
the solvability of $\cL_m$ follows from the next lemma, completing the 
proof of the Theorem.  
\end{proof}
\begin{Lemma}
\label{solLiealg}
A holomorphic Lie algebra $\cL$ generated by homogeneous vector fields 
of degree $ \geq 0$
 is solvable if and only if the degree $0$ subalgebra is solvable.
\end{Lemma}
\begin{proof}  Let $L_0$ denote the Lie algebra of vector fields of degree 
zero (it is a linear Lie algebra). Also, let $\cL^{(k)}$ denote the 
holomorphic sub-Lie algebra generated by the homogeneous vector fields 
of degree $\geq k$.  Then, as $[\cL^{(k)}, \cL^{(j)}] \subset \cL^{(k+j)}$, it 
follows that the Lie algebra $\cL^{(k)}/ \cL^{(k+1)}$ is abelian for $k \geq 
1$.  Lastly, the projection induces an isomorphism $\cL/ \cL^{(1)} \, =\, 
\cL^{(0)}/ \cL^{(1)} \simeq L_0$.  This is solvable by assumption.  Hence, if 
we adjoin to $\{ \cL^{(k)}\}$ the pullback of the derived series of $L_0$ 
via the projection of $\cL$ onto $L_0$, we obtain a filtration by 
subalgebras, each an ideal in the preceding, whose successive quotients 
are abelian.  Hence, $\cL$ is solvable.  \par
For the reverse direction we just note that $L_0$, as a quotient of the 
solvable Lie algebra $\cL$, is solvable.
\end{proof}

\section{Block Representations by Restriction and Extension}
\label{S:RestrBlocReps}
\par
In this section we apply the restriction and
extension properties of block representations to obtain free divisors 
which will be used in part II. \par
Suppose $\rho : G \to \GL(V)$ is a block representation with associated 
decomposition
$$   V = W_k \supset W_{k-1} \supset \cdots  \supset W_{1}  \supset
W_{0} = (0)  $$
 with $K_j = \ker(\rho_j)$ for the induced representation $\rho_j : G \to
\GL(V/W_j)$. \par
If we restrict to the representation of $K_m$ on $W_m$, we will obtain a 
decomposition descending from $W_m$ with corresponding normal 
subgroups $K_j$.  We already know that the resulting coefficient matrix 
has the necessary block triangular form.  
There is a problem because the corresponding relative
coefficient determinants are those for $\rho$ restricted to the
subspace $W_m$.  Although the relative coefficient determinants were
reduced and relatively prime as polynomials on $V$, this may not
continue to hold on $W_m$.  \par
A simple example illustrating this problem occurs for the bilinear form
representation of $B_2$ on $\Sym_2(\C)$.  Suppose we restrict to the 
subspace
$W_1 \subset \Sym_2(\C)$ of symmetric matrices with upper left entry $ 
= 0$.  
The corresponding normal subgroup of $B_2$ has upper left entry $ = 1$.  
In terms of the basis used in  \S \ref{S:CholFact}, the coefficient matrix 
is
$ A = \begin{pmatrix}
0 & a_{1\, 2} \\ a_{1\, 2}  &  a_{2\, 2}
\end{pmatrix}$.
Thus, the relative coefficient matrix is $a_{1\, 2}^2$, so it is a
nonreduced block representation.  \par
Nonetheless, in many cases of interest we may restrict a tower
of block representations by modifying the lowest degree one to obtain 
another tower of block representations. 

\subsection*{Restricted Symmetric Representations} \hfill
\par 
We consider several restrictions of the tower of representations $\{
(B_m, \Sym_m)\}$.   First, for the subrepresentations $\{ (G_m, W_{m-
1})\}$ for $m\geq 3$.  Here $G_m \subset B_m$ is the subgroup of 
matrices $B = (b_{i\,
j}) \in B_m$ with entries $b_{2\, 1} = 0$ so that the upper left $3 \times
3$-block has the form a) in (\ref{Eqn12.5}).
\begin{equation}
\label{Eqn12.5}
a) \quad \begin{pmatrix}
*  & 0  & 0  \\
0  & *  & 0  \\
 * & *  &*
\end{pmatrix}  \qquad b)  \quad
\begin{pmatrix}
*  & 0  & 0  &0   \\
0  & *  & 0  &0   \\
0  & * & *  &0   \\
0  & * & *  &*  &
\end{pmatrix} \,\, .
\end{equation}
As in \S \ref{S:CholFact}, we let $W_{m-1} = \Sym_m^{(m-1, m-1)}(\C) 
\subset \Sym_m(\C)$, which is the subspace of symmetric matrices with  
the upper left entry equal to $0$.  With the same inclusions as for 
$\{ (B_m, \Sym_m(\C))\}$, $\{ (G_m, W_{m-1})\}$ is again a tower of 
representations.  
\par
Second, we consider the restriction of the same tower
$\{ (B_m, \Sym_m(\C))\}$ but to the subspace $W_{m-2}$, which consists 
of matrices with the upper left hand $2 \times 2$ block equal to $0$.  We
only consider the tower beginning with $m \geq 4$.  This time we choose
$G_m$ to be the subgroup of $B_m$ consisting of matrices with upper left
$4 \times 4$-block of the form b) in (\ref{Eqn12.5}).
\begin{Proposition}
\label{Prop12.4}
The two restrictions of the tower $\{ (B_m, \Sym_m(\C))\}$ define block 
representations of towers.  Thus, the exceptional orbit varieties are free 
divisors and have defining equations given by: for the first case
\begin{equation}
\label{Eqn12.1a}
  - a_{1\, 2}\,  a_{2\, 2} \cdot (a_{3\, 3} a_{1\, 2}^2  -
2  a_{2\, 3} a_{1\, 2} a_{1\, 3}  +   a_{2\, 2} a_{1\, 3}^2)\, \cdot \prod_{k 
= 4}^{m} \det(A^{(k)}_{1}) \, \, = \,\, 0\, ; 
\end{equation}
and for the second case 
\begin{equation}
\label{Eqn12.1b}
  - a_{1\, 3}\, a_{2\, 3}\cdot  (a_{1\, 3}a_{2\, 4}  - a_{1\, 4} a_{2\,
3})\cdot ( a_{3\, 3}a_{2\, 4}^2  - 2 a_{3\, 4} a_{2\, 4} a_{2\,
3}  + a_{4\, 4} a_{2\, 3}^2) \, \cdot \prod_{k = 5}^{m} \det(A^{(k)}_{2}) \, \, 
= \,\, 0\, . 
\end{equation}
where $A^{(k)}_{r}$ denotes the upper left $k \times k$ submatrix of 
$A_{r}$, which is obtained from $A$ by setting $a_{i, j} = 0$ for $1 \leq i, 
j \leq r$. 
\end{Proposition}
\begin{Remark}
The middle term in (\ref{Eqn12.1a}) is the determinant of the generic
$3 \times 3$ symmetric matrix with $a_{1\, 1} = 0$ and for 
(\ref{Eqn12.1b}) it is minus the determinant of the $3
\times 3$ lower-right submatrix of $A^{(4)}_{1}$ (so $a_{2\, 2} = 0$), and 
it is reduced.  \par
\end{Remark}
\begin{proof}
The proof of each statement is similar so we just consider the second 
case.
It is the restriction of the tower $\{ (B_m, \Sym_m(\C))\}$ to the 
subspace 
$W_{m-2}$, which consists of matrices with the upper left hand $2 
\times 2$ block equal to $0$.  Then, we will apply the Restriction 
Property, Proposition \ref{prop5.2}.  \par
It is only necessary to consider the diagonal block corresponding to
$W_{m-2}/W_{m-4}$ and $G_m/K_{m-4}$.  It is sufficient to consider the
subrepresentation on $W_2 \subset \Sym_4(\C)$. We use the 
complementary
bases
$$  \{E_{1\, 1},  E_{2\, 2}, E_{3\, 2},
E_{3\, 3}, E_{4\, 2}, E_{4\, 3}, E_{4\, 4}\} \quad \makebox{to $\bk_{m-4}$
in $\bg_m$, and } $$
$$  \{ e_{1\, 3}, e_{2\, 3}, e_{3\, 3}, e_{1\, 4}, e_{2\, 4}, e_{3\, 4}, e_{4\,
4}\} \quad \makebox{to  $W_{m-4}$ in $W_{m-2}$} $$
(using the notation of \S  \ref{S:CholFact}).  \par
The corresponding relative coefficient matrix has the form
$$ \begin{pmatrix}
a_{1\, 3}  & 0  & 0  & a_{1\, 3}  &   0       &      0        &           0          \\
0  & a_{2\, 3} &  0  &  a_{2\, 3} &    0       &      0       &           0          \\
0  &  0 &  a_{2\, 3}  &  a_{3\, 3} &    0      &      0        &          0           \\
a_{1\, 4}  & 0   &   0  &     0     &        0     &   a_{1\, 3}  &    a_{1\, 4}   \\
0 &   a_{2\, 4}  &  0   &   0       &        0      &   a_{2\, 3}  &    a_{2\, 3}  \\
0 &  0   &  a_{2\, 4}  & a_{3\, 3}  &  a_{2\, 3}  &  a_{3\, 3}  &   a_{3\, 4}
\\
0  &    0    &       0     &     0          &  a_{2\, 4}  &       0     &      a_{4\, 4}
\end{pmatrix}  $$
This has for its determinant the reduced polynomial
$$  - a_{1\, 3}\, a_{2\, 3}\cdot  (a_{1\, 3}a_{2\, 4}  - a_{1\, 4} a_{2\,
3})\cdot ( a_{3\, 3}a_{2\, 4}^2  - 2 a_{3\, 4} a_{2\, 4} a_{2\,
3}  + a_{4\, 4} a_{2\, 3}^2) \, . $$  
\par
Then, the subsequent relative coefficient determinants are those for
$(B_m, \Sym_m(\C))$, but with $a_{1\, 1} = a_{1\, 2} = a_{2\, 2} = 0$.  
Just 
as for the unrestricted case, we see using Lemma \ref{Lem7.4} that they 
are reduced  and relatively prime.  Hence, we obtain a tower of block 
representations.  Thus, the exceptional orbit variety is free with defining 
equation the product of the relative coefficient determinants.
\end{proof}

\subsection*{Restricted General Representations} \hfill
\par  
We second consider the restrictions of the tower of block representations 
formed from  $(B_m \times C_m, M_{m, m})$ and $(B_{m-1} \times C_{m} 
, M_{m-1, m})$ as in \S \ref{S:CholFact}.  These together form a tower
of block representations.  We consider the restriction to
the subspaces where $a_{1, 1} = 0$ for $m \geq 3$.  We replace $B_m$ by 
the subgroup $B_m^{\prime}$ with upper left hand $2 \times 2$ matrix a 
diagonal matrix.  \par
\begin{Proposition}
\label{Prop12.5}
For restrictions of the tower formed from $(B_m \times C_m, M_{m, m})$ 
and $(B_{m-1} \times C_{m} , M_{m-1, m})$ define block representations 
of towers so the exceptional orbit varieties are free divisors and have 
defining equations given by: for $M_{m, m}$ with $m \geq 3$, 
\begin{equation}
\label{Eqn12.2a}
  a_{1\, 2}\, a_{2\, 1}\,  a_{2\, 2}\cdot (a_{1\, 2} a_{2\, 3} - a_{1\, 3} 
a_{2\, 2}) \, \cdot \prod_{k =3}^{m}
\det(A^{(k)}_{1})\cdot \prod_{k =3}^{m-1}\det(\hat A^{(k)}_{1}) \,\, = \,\, 0 
\, ; 
\end{equation}
and for $M_{m-1, m}$, with $m \geq 3$,
\begin{equation}
\label{Eqn12.2b}
 a_{1\, 2}\, a_{2\, 1}\,  a_{2\, 2}\cdot (a_{1\, 2} a_{2\, 3} - a_{1\, 3} 
a_{2\, 2}) \,  \cdot \prod_{k =3}^{m-1} \det(A^{(k)}_{1})\cdot \prod_{k 
=3}^{m-1}\det(\hat A^{(k)}_{1}) \,\, = \,\, 0 
\end{equation}
with $A^{(k)}_{1}$ as defined earlier. 
\end{Proposition}
\begin{proof}
The proof is similar to that for Proposition \ref{Prop12.4}.  It is 
sufficient to consider the (lowest degree) representation of $G_2 = 
B_2^{\prime} \times C_{3}$ on $M_{2, 3}$, and then restrict the other 
relative coefficients determinants by evaluating those from Theorem 
\ref{ModCholRep} with $a_{1, 1} = 0$ and use Lemma \ref{Lem7.4} to see 
that they are reduced and relatively prime.  \par
  We compute the coefficient matrix using the complementary bases
$$  \{(E_{1\, 1}, 0),  (E_{2\, 2}, 0), (0, E_{2\, 2})\} \quad \makebox{to 
$\bk_{m-2}$ in $\bg_m$, and } $$
$$  \{ E_{1\, 2}, E_{2\, 1}, E_{2\, 2}\} \quad \makebox{to  $W_{m-2}$ in 
$W_{m-1}$}\, . $$
The corresponding coefficient determinant will be, up to sign, 
\begin{equation*}
a_{1\, 2}\, a_{2\, 1}\,  a_{2\, 2}\cdot (a_{1\, 2} a_{2\, 3} - a_{1\, 3} 
a_{2\, 2}) \, .\qedhere
\end{equation*}
\end{proof}
\par
The preceding involve restrictions of block representations of solvable 
linear algebraic groups.   We may also apply the Extension Property, 
Proposition \ref{prop5.3}, to extend block representations for a class of 
groups which extend both solvable and reductive groups.

\begin{Example}[Extension of a solvable group by a reductive group]
\label{ex:extension}
We consider the restriction of the bilinear form representation to the 
group
$$
G_3=\left\{\begin{pmatrix} * & 0 & 0 \\ 0 & * & * \\ 0 & * & *
\end{pmatrix}\in\GL_3(\C)\right\}
$$
and to the subspace $V_3 = \Sym_3^{(2, 2)}(\C) \subset \Sym_3(\C)$, 
consisting of matrices with upper left entry zero.  We considered the 
restriction to this subspace in Proposition \ref{Prop12.4}; however, now 
the group $G_3$ is reductive.  This representation also will play a role in 
part II in the computations for $3 \times 3$ symmetric matrix 
singularities.  A direct calculation shows that this equidimensional 
representation has coefficient determinant 
$$ - (a_{2\,2}a_{3\,3}-a_{2\,3}^2)\cdot (a_{3\, 3} a_{1\, 2}^2  -
2  a_{2\, 3} a_{1\, 2} a_{1\, 3}  +   a_{2\, 2} a_{1\, 3}^2) \, , 
$$
which defines the exceptional orbit variety as a linear free divisor on 
$V_3$.  
The second term in the product is the determinant of the $3 \times 3$ 
matrix with $a_{1\, 1} = 0$.  \par
The Extension Property, Proposition \ref{prop5.3}, now allows us to 
inductively extend the reductive group $G_3$ by a solvable group, and the 
representation to a representation of the extended group, obtaining a 
linear free divisor for the larger representation.  We again use the 
notation of \S \ref{S:CholFact}.
For $m\geq 3$, we more generally let $V_m = \Sym_m^{(m-1, m-1)}(\C) 
\subset \Sym_m(\C)$ (also the subspace considered in Proposition 
\ref{Prop12.4}).
However, the extended group
\begin{equation}
\label{Eqn12.2}
G_m=\left\{ \begin{pmatrix} A & 0 \\ B & C \end{pmatrix}\in \GL_m(\C):
A\in G_3, C\in B_{m-3}(\C)\right\}
\end{equation}
is no longer reductive (nor solvable).   We note that it is the extension of 
$G_3$ by the solvable subgroup $K_{m-3}$ consisting of elements in 
$G_m$ with $A = I$ in (\ref{Eqn12.2}).  These subgroups were used earlier 
in both \S \ref{S:CholFact} for the tower structure of $\Sym_m(\C)$ and 
also in Proposition \ref{Prop12.4}.  Then, $(G_m , V_m)$ for the bilinear 
form representation restricts to $G_m$ acting on $V_m$ form a tower of 
representations using the same inclusions (\ref{Eqn9.2a}) as earlier.  \par
\begin{Proposition}
\label{Prop12.6}
The $\{(G_m , V_m)\}$ for $m\geq 3$ form a tower of block 
representations so the exceptional orbit varieties are linear free divisors 
and their defining equations are given by 
\begin{equation}
\label{eqn:x223323}
(a_{2\,2}a_{3\,3}-a_{2\,3}^2)\cdot \prod_{j=3}^m \det(A^{(j)}_{1}) \,\, = 
\,\, 0 \, .
\end{equation}
\end{Proposition}
\par
\begin{proof}
To verify this claim, we apply the extension property to the entire tower 
in the form of Proposition \ref{Prop7.2}.  The first group and 
representation are $(G_3, V_3)$ which is a block representation with just 
one block.  \par
Next, we let $W_{1} =  \Sym_m^{(1, 1)}(\C) \subset V_m$.  
The kernel of the quotient representation $G_m\to\GL(V_m/W_1)$
is the product of a finite group with the subgroup $K_{1}\subset G_m$.  
Then, $G_m/K_1$ is naturally identified with $G_{m-1}$, and $V_m/W_1$ 
with $V_{m-1}$.  With these identifications, 
$G_m/K_m\to\GL(V_m/W_m)$ is isomorphic as a representation to 
$G_{m-1}\to\GL(V_{m-1})$.  This establishes ii) of Proposition  
\ref{Prop7.2}.  \par
Lastly, the coefficient determinant for $K_1$ acting on $V_m$ with 
$a_{1, 1} = 0$ is $\det(A^{(m)}_{1})$.  As this is not identically zero, 
$K_1$ has a relatively open orbit.  Also, this polynomial is irreducible and 
relatively prime to the coefficient determinant for $G_m/K_1$.
Thus, ii) of Proposition \ref{Prop7.2} follows and the claim for $(G_m, 
V_m)$ follows.  
\end{proof}
\end{Example}
\par
It appears that linear free divisors can often be extended to larger linear 
free divisors using an extension of the original group by a solvable
group.  For more examples see \cite[\S 5.3]{P}.

\end{document}